\documentclass[11pt,reqno]{amsart}
\usepackage{amsthm, amsfonts, amssymb, color}
 \usepackage{mathrsfs}
\usepackage{enumerate}
\usepackage{amsmath}
 \usepackage{amstext, amsxtra}
  \usepackage{txfonts}
 \usepackage[colorlinks, linkcolor=black, citecolor=blue, pagebackref, hypertexnames=false]{hyperref}
\usepackage{graphicx}
\usepackage[symbol]{footmisc}
\usepackage{cite}
 \allowdisplaybreaks
\def\R {\mathbb{R}}

\def\Re {\mathfrak{Re\,}}

\def\d{{\rm d}}

\def\i{{\rm i}}

\def \and {{\qquad\text{and}\qquad}}

\newtheorem{theorem}{Theorem}[section]
\newtheorem{proposition}[theorem]{Proposition}

\newtheorem{lemma}[theorem]{Lemma}

\newtheorem{example}[theorem]{Example}
\newtheorem{remark}[theorem]{Remark}

\numberwithin{equation}{section}
\theoremstyle{definition}

\title
 {Observable sets, potentials and Schr\"{o}dinger equations}
\author{Shanlin Huang \quad Gengsheng Wang \quad   Ming Wang}

\address{Shanlin Huang,  School of Mathematics and Statistics, Hubei Key Laboratory of Engineering Modeling and Scientific Computing, Huazhong University of Science and Technology,  Wuhan,  430074,  P.R. China}
\email{shanlin\_huang@hust.edu.cn}
\address{Gengsheng Wang,  Center for Applied Mathematics, Tianjin University, Tianjin 300072, P.R. China}
\email{wanggs62@yeah.net}
\address{Ming Wang,  School of Mathematics and Physics, China University of Geosciences, Wuhan 430074,  P.R. China}
\email{mwang@cug.edu.cn}
\thanks{Corresponding author: G. Wang, e-mail: wanggs62@yeah.net}

\begin{document}

\begin{abstract}
We characterize observable sets for 1-dim Schr\"{o}dinger equations in $\mathbb{R}$: $\i \partial_t u = (-\partial_x^2+x^{2m})u$ (with $m\in \mathbb{N}:=\{0,1,\dots\}$). More precisely, we obtain
what follows: First, when $m=0$, $E\subset\mathbb{R}$ is an observable set at some time  if and only if it is thick,  namely, there is  $\gamma>0$ and $L>0$ so that
$$
\left|E \bigcap [x, x+ L]\right|\geq \gamma L\;\;\mbox{for each}\;\;x\in \mathbb{R};
$$
 Second, when $m=1$ ($m\geq 2$ resp.), $E$ is an observable set at some time (at any time resp. )   if and only if it is weakly thick, namely
$$
\varliminf_{x \rightarrow +\infty} \frac{|E\bigcap [-x, x]|}{x} >0.
$$
From these,  we see how   potentials $x^{2m}$ affect the observability (including the geometric structures of observable sets and the minimal observable time).

Besides, we obtain several supplemental theorems for the above results, in particular,
we find
 that
 a half line  is an observable set at  time $T>0$
 for the above equation with $m=1$
 if and only if $T>\frac{\pi}{2}$.
\end{abstract}

\maketitle

\section{Introduction}\label{section1}
\noindent

Consider two $1$-dim Schr\"{o}dinger equations in $\R$. The first one is as:
\begin{equation}\label{equ-free-sch}
\i \partial_t u(t,x) = (-\partial_x^2+c) u(t,x),\;\; t\in \mathbb{R}^+:=(0,\infty),\; x\in \mathbb{R}; \quad u(0,\cdot)\in L^2(\R),
\end{equation}
where $c$ is a real number, while the second one reads as:
\begin{equation}\label{equ-1.0}
\i \partial_t u(t,x) = H u(t,x),\;\; t\in \mathbb{R}^+,\; x\in\mathbb{R}; \quad u(0,\cdot)\in L^2(\R),
\end{equation}
where
\begin{equation}\label{equ-1.0-1130}
H:=-\partial_x^2+x^{2m},\quad m\in \mathbb{N}^+:=\{1,2,\ldots\}.
\end{equation}
(Here and in what follows, $L^2(\R):=L^2(\R;\mathbb{C})$. The same is said about $L^2(\R^n)$.)
 Several notes on these two equations are given in order.
\begin{itemize}
\item When $c=0$,  \eqref{equ-free-sch} is referred to as
the free Schr\"{o}dinger equation, while when $c=1$, it is the equation \eqref{equ-1.0}
with $m=0$.

\item  The equation \eqref{equ-1.0}, with $m=1$, is known as the Hermite Schr\"{o}dinger equation. It is  the  Schr\"{o}dinger equation for the harmonic oscillator, which is a basic model in quantum mechanics, as it approximates any trapping Schr\"{o}dinger equation with a real potential at its point of equilibrium (see e.g. in \cite{FH,gri}).

\item Both  \eqref{equ-free-sch} and \eqref{equ-1.0}
 are well-posed in  $L^2(\R)$. Furthermore,
the  $L^2$-norm of  any solution $u$ to \eqref{equ-free-sch} (or \eqref{equ-1.0})
 is conserved, i.e.,
    \begin{align}\label{equ-929-2}
\int_{\mathbb{R}}|u(t,x)|^2\,\mathrm dx = \int_{\mathbb{R}}|u(0,x)|^2\,\mathrm dx\;\;
\mbox{for any}\;\; t \in \R^+.
\end{align}

    \end{itemize}

\subsection{Several definitions}\label{sev def}

Let $E\subset\mathbb{R}$ be a measurable set. Several definitions on $E$ are given as follows:

 \begin{itemize}
\item [($\textbf{D}_1$)] The set $E$  is called \emph{an observable set at any  time} for
\eqref{equ-free-sch} (or \eqref{equ-1.0}), if for any $T>0$, there is  a constant $C_{obs}=C_{obs}(T,E)>0$ so that
\begin{align}\label{ob-q}
\int_{\mathbb{R}}|u(0,x)|^2\,\mathrm dx\leq C_{obs}\int_0^T\int_E|u(t,x)|^2\,\mathrm dx\, \mathrm dt,\;\mbox{when}\;u\;\mbox{solves}\; \eqref{equ-free-sch}\; (\mbox{or}\,  \eqref{equ-1.0}).
\end{align}
(Here and in what follows, $C(\cdots)$ stands for a positive constant depending on what are enclosed in the brackets.)

\item [($\textbf{D}_2$)] The set $E$ is called \emph{an observable set   at some  time} for  \eqref{equ-free-sch} (or \eqref{equ-1.0}), if there is $T>0$ and $C_{obs}=C_{obs}(T,E)>0$ so that \eqref{ob-q} holds. Similarly, the set $E$ is called \emph{an observable set   at   time $T>0$} for  \eqref{equ-free-sch} (or \eqref{equ-1.0}), if there is  $C_{obs}=C_{obs}(T,E)>0$ so that \eqref{ob-q} holds.

 \item [($\textbf{D}_3$)]   The set $E$ is said to be \emph{thick}, if there is  $\gamma>0$ and $L>0$ so that
\begin{align}\label{equ-set}
\left|E \bigcap [x, x+ L]\right|\geq \gamma L\;\;\mbox{for each}\;\;x\in \mathbb{R}.
\end{align}
(Here and in what follows, when $E$ is a measurable set in $\mathbb{R}$ (or in $\mathbb{R}^n$), $|E|$ stands for the Lebesgue measure of
$E$ in $\mathbb{R}$ (or in $\mathbb{R}^n$).)

\item [($\textbf{D}_4$)]  The set $E$ is said to be \emph{weakly thick},
if
\begin{align}\label{def-1-3}
\varliminf_{x \rightarrow +\infty} \frac{|E\bigcap [-x, x]|}{x} >0.
\end{align}

\end{itemize}
Several notes on these definitions are given in order.
\begin{itemize}
\item The above definitions can be extended to the $n$-dim case similarly.
\item The inequality \eqref{ob-q} is the standard observability inequality for \eqref{equ-free-sch} (or \eqref{equ-1.0}). Thus,  $E\subset\mathbb{R}$ is \emph{ an observable set   at some time} for  \eqref{equ-free-sch} (or \eqref{equ-1.0}) if and only if \eqref{equ-free-sch} (or \eqref{equ-1.0}), with controls restricted in $E$, is exactly controllable over $(0,T)$ for some $T>0$, while $E\subset\mathbb{R}$ is  \emph{an observable set   at any  time} for  \eqref{equ-free-sch} (or \eqref{equ-1.0}) if and only if \eqref{equ-free-sch} (or \eqref{equ-1.0}), with controls restricted in $E$, is exactly controllable over $(0,T)$ for any $T>0$.

    \item To our best knowledge, the concept of \emph{thick sets}  arose from studies of the uncertainty principle (see, for instance, \cite[p. 5]{BD}, or \cite[p. 113]{HJ}),
 while  the concept of \emph{weakly thick sets} seems to be insufficiently explored. We notice that in the recent work \cite{BJ}, the authors have used concepts \eqref{equ-set} and \eqref{def-1-3} (in $\mathbb{R}^n$) to study spectral inequalities for Hermite functions.
        We also mention that if $\varliminf_{}$ is replaced by $\varlimsup_{}$ in the above \eqref{def-1-3}, then some interesting geometric properties of such sets (in $\mathbb{R}^n$) can be found in \cite{B86}.
        The relationship between \emph{thick sets} and \emph{weakly thick sets} is established in Proposition \ref{prop-916} and Proposition \ref {prop-916-1} of this paper, which tells us that   if $E$ is \emph{thick}, then it is \emph{weakly thick}, but a  weakly thick  set may not  be a thick set.
                 We can explain the difference between \emph{thick sets} and \emph{weakly thick sets} in the manner:  \emph{A thick set} is ``thick" in any interval with a fixed length,
    while   \emph{a weakly thick set} is ``thick" in   the  interval $[-x, x]$ with $x$ sufficiently large.

    \end{itemize}

    \subsection{Aim and motivation}\label{aim and moti}
    The aim of this  paper is to  present characterizations of \emph{observable sets}
     (at some time or at any time) for equations \eqref{equ-free-sch} and \eqref{equ-1.0}.
      From these, we can see how potentials  affect  the observability (including the geometric structures of observable sets and the minimal observable time).
    Our studies were partially motivated by the existing fact:  different  potentials may cause different geometric structures of observable sets for heat equations in $\mathbb{R}^n$,
    with $n\in \mathbb{N}^+$. More precisely, consider
                  the
    following  heat equations
    in $\mathbb{R}^n$:
    \begin{align}\label{pureheatequ}
\partial_t u(t,x) -\Delta u(t,x)=0,\;\;(t,x)\in \mathbb{R}^+\times\mathbb{R}^n,  \quad u(0, x)\in L^2(\R^n)
 \end{align}
  and
    \begin{align}\label{equ-102-1}
    \partial_t u(t,x)-\Delta u(t,x)+|x|^{2m}u(t,x)=0, \;\;(t,x)\in \mathbb{R}^+\times\mathbb{R}^n,\quad u(0, x)\in L^2(\R^n),
    \end{align}
    where $m\in \mathbb{N}^+$.
     First, it was shown independently in \cite{EV18, WWZZ} that $E$ is  \emph{an observable set
     at any time} for the heat equation \eqref{pureheatequ} if and only if $E$ is thick in $\mathbb{R}^n$, i.e., for some $\gamma>0$
 and $L>0$,
\begin{align}\label{equ-set-n}
\left|E \bigcap Q_L(x)\right|\geq \gamma L^n\;\;\mbox{for each}\;\;x\in \mathbb{R}^n.
\end{align}
(Here $Q_L(x)$ is the closed cube in $\mathbb{R}^n$, centered at $x$ and of the length $L$.)
Second, it was obtained in  \cite{M09, DM} that when $m\geq 2$,
the cone
$E=\{x\in \mathbb{R}^{n}: |x|\geq r_0, x/|x|\in \Theta_0\}$ (where $r_0> 0$ and  $\Theta_0$
 is a nonempty and open subset of $\mathbb{S}^{n-1}$) is an observable set for \eqref{equ-102-1},
 while when $m=1$, the above cone is no longer  an observable set for \eqref{equ-102-1}.
 Third, it was proved in
  \cite{BJ} that if $E$ is thick, then $E$ is an observable set for \eqref{equ-102-1} with $m=1$.

\subsection{Main results}\label{main results,1.3}
The first main result of this paper is as follows:
\begin{theorem}\label{thm-ob-for-free}
Let $E\subset \R$ be a measurable set. Then the following  statements are equivalent:

\noindent $(i)$  The set $E$ is  thick.

\noindent $(ii)$  The set $E$ is  an observable set  at some time  for the equation \eqref{equ-free-sch}  with $c\in \R$.

\noindent $(iii)$  The set $E$ is  an observable set  at some time for the equation \eqref{equ-free-sch} with $c=0$.
\end{theorem}
Some remarks are given in order.

\begin{itemize}
\item [($\textbf{a}_1$)]   Theorem \ref{thm-ob-for-free} characterizes  the \emph{observable set at some time} for the Schr\"{o}dinger equation \eqref{equ-free-sch}. This characterization seems new for us, though some sufficient conditions on  \emph{observable sets} for Schr\"{o}dinger equations in $\R^n$ were built up in  \cite{WWZ,zhang09}. We would like to mention some sufficient conditions mentioned above: According to Remark (a6) in \cite{WWZ},  $E$ is an \emph{observable set  at any time} for the free Schr\"{o}dinger equation,  if $E$ contains  $B^c(0,r)$ for some $r>0$. (Here and in what follows, $B(x_0,r)$ denotes the closed ball in
    $\mathbb{R}^n$, centered at $x_0\in\mathbb{R}^n$ and of radius $r>0$, while $B^c(x_0,r)$ denotes its complementary set.)
    The same conclusion was derived in \cite{zhang09}  for the Schr\"{o}dinger equation with Schwartz class potentials.

\item [($\textbf{a}_2$)]  For the Schr\"{o}dinger equation on compact Riemannian manifolds, the  observability has been extensively studied: It was shown in  \cite{Le92} (see also \cite{KDPhung})
           that any open set with geometric control
condition (GCC) is \emph{an observable set at any time}.
    It was further proved in \cite{Ma11} that  the GCC  is also necessary in the manifolds with periodic geodesic flows (or in Zoll manifolds). It was obtained that on the flat torus $\mathbb{T}^n := (\R/2\pi\mathbb{Z})^n$, every non-empty open set $E\subset \mathbb{T}^n$ is \emph{an observable set at any time} (see, for instance,  \cite{Ha,J90,Kom} for the free Schr\"{o}dinger equation and \cite{AM,BZ12,Bour13} for the Schr\"{o}dinger equations with potentials). It was verified in \cite{BZ17} that  on $\mathbb{T}^2$, each measurable set $E$, with a positive measure, is \emph{an observable set at any time}.
    For the observability of Schr\"{o}dinger equations on negatively curved manifolds, we refer the readers to  \cite{AR,Dya,Jin} and the references therein.

\end{itemize}

The second main result of this paper is as follows:
\begin{theorem}\label{thm-ob-for-HO}
Let $E\subset \R$ be a measurable set. Then the following  statements are equivalent:

\noindent $(i)$ The set $E$ is  weakly thick.

\noindent $(ii)$  The set  $E$ is  an observable set at some time  for the equation  \eqref{equ-1.0} with  $m=1$.

\noindent $(iii)$  The set $E$ is  an observable set at any time  for the equation  \eqref{equ-1.0} with  $m\ge 2$.
\end{theorem}
Several notes are given in order.
\begin{itemize}
\item [($\textbf{b}_1$)] Theorem \ref{thm-ob-for-HO}  characterizes \emph{observable sets} for the equation  \eqref{equ-1.0}.
    In this direction, we would like to mention \cite{DM} which shows that
    a half line: $(-\infty,\, x_0)$ (or $(x_0,\, \infty)$), with $x_0\in \R$, is  \emph{an observable set at some time} for the equation \eqref{equ-1.0} with $m=1$ (see \cite[Proposition 3]{DM}). Compared to this, our  Theorem \ref{thm-ob-for-HO} shows  that
    the   much smaller set:
      $$
      E = \bigcup_{k=1}^\infty [kx_0, (k+{1}/{2})x_0]\;\;\mbox{with}\;\; x_0>0
      $$
    is also  \emph{an observable set at some time} for the equation \eqref{equ-1.0} with $m=1$.
    With regard to \emph{observable sets} for  the equation \eqref{equ-1.0}, where $m\geq 2$,
    we  do not find any result   in the existing literatures.

 \item [($\textbf{b}_2$)] Since \emph{a weakly thick set} may not be \emph{a thick set} (see Subsection \ref{sev def}), it follows from Theorem \ref{thm-ob-for-free} and Theorem \ref{thm-ob-for-HO} that the characterizations of \emph{observable sets}  for the free Schr\"{o}dinger equation and the equation \eqref{equ-1.0} are quite different.
          The reason  is as follows: Differing from  the operator $(-\partial_x^2+c)$ (which has  the continuous spectrum), the operator $H=-\partial_x^2+x^{2m}$  has purely discrete spectrum consisting of all simple and real eigenvalues $\{\lambda_k\}_{k=1}^\infty$, with a gap condition  (see \eqref{equ-929-11}).
     This gap condition ensures that $E$ is  \emph{an observable set at some time} if and only if $\|\varphi_k\|_{L^2(E)}$ has a uniform positive lower bound for all $k\in \mathbb{N}^+$,
     where $\varphi_k$ is the $L^2$ normalized eigenfunction corresponding to $\lambda_k$ (see Proposition \ref{pro-ob-general}).  Furthermore,   we observed that each eigenfunction $\varphi_k$ is either even or odd (see {\bf Key Observation} in Subsection \ref{subsect4.2}). This fact suggests that  \emph{the observable set}  can be chosen only on a half line, which is not a thick set clearly.

      \item [($\textbf{b}_3$)]  Theorem \ref{thm-ob-for-HO} shows the difference between
     the equation \eqref{equ-1.0} with $m=1$ and  $m\geq 2$ respectively, from the perspective of the observability. It is natural to ask if  one can replace
     \emph{at some time} by  \emph{at any time} in the statement  \noindent $(ii)$ of Theorem \ref{thm-ob-for-HO}. The answer is negative. In fact, on one hand, the half line
      $(a, \infty)$ (with $a\in\mathbb{R}$) is clearly \emph{a weakly thick set} (see Example \ref{exp-1}), while on the other hand,  $(a, \infty)$ is  \emph{an observable set at time $T$} if and only if $T>\frac{\pi}{2}$ (see Theorem \ref{thm-1126-1}).
       Hence, the above difference is essential.
       The reason behind this phenomenon is closely related to the different asymptotic behaviours  of the associated spectral distributions: in the case that  $m\ge 2$,  we have (see \eqref{equ4.34} in Section \ref{sec4})
$$
 \lambda_{k+1}-\lambda_k\rightarrow \infty, \quad \mbox{as} \quad k\rightarrow \infty,
$$
while in the case when $m=1$, we have (see  \eqref{equ1.3}  in Section \ref{two-point-hermite})
 $$
 \lambda_{k+1}-\lambda_k\equiv 2\;\;\mbox{for all}\;\;k\in \mathbb{N}^+.
 $$

 \item [($\textbf{b}_4$)] Our strategy to prove Theorem \ref{thm-ob-for-HO} is as follows:
  First, we  treat a class of more general  potentials which are real-valued, have the $C^3$-regularity and   grow at infinity like $|x|^{2c}$, with $c\geq 1$ a real number (see the \textbf{Condition (H)} in Section \ref{sec3} for details). For such a potential, we prove in Theorem \ref{thm-ob-for-genral} that $E$ is  \emph{an observable set at some time}, if
\begin{align}\label{equ-2-25-1}
\varliminf_{x \rightarrow +\infty} \frac{|E\bigcap [0, x]|}{x} >0.
\end{align}
Second,  we observe that when $V(x)=x^{2m}$ (with $m\in \mathbb{N}^+$),  each  eigenfunction $\varphi_k$ of $H$ is either even or odd. With the aid of this observation, we  prove
 that  \emph{a weakly thick set}  is \emph{an observable set} for \eqref{equ-1.0}. Third, with the help of the asymptotic formulas for eigenvalues (see \eqref{equ4.34}) and the asymptotic expressions of eigenfunctions (see Lemma \ref{lem-append}), we prove  that \emph{an observable set} for \eqref{equ-1.0} is \emph{weakly thick}.
\end{itemize}

\subsection{Supplemental results}\label{subsec suplement}
This subsection gives three theorems on  the observability for the Hermite Schr\"{o}dinger equation in
$\mathbb{R}^n$  (with $n\in \mathbb{N}^+$):
\begin{align}\label{equ1.1}
\i\partial_tu(t, x)=Hu(t, x),\;\; (t,x)\in  \mathbb{R}^+\times \mathbb{R}^n;\,\,\,\,\,u(0,\,\cdot)=u_0(\cdot)\in L^2(\mathbb{R}^n),
\end{align}
where $H$ is the Hermite operator:
\begin{align}\label{equ1.2}
H=-\Delta+|x|^2.
\end{align}
(Here and in what follows, $|x|$ stands for the Euclidean norm of $x\in \mathbb{R}^n$.
We write $e^{-\i tH}$ for the unitary group generated by $-\i H$.)
These theorems can be   viewed  as supplemental results  of Theorem \ref{thm-ob-for-HO}.

\begin{theorem}\label{thm-1126-2}
Given $x_0\in \mathbb{R}^n$ and $r>0$,  the exterior domain $B^c(x_0,\, r)$ is an observable set at any time for  \eqref{equ1.1}. Furthermore,  for any $T>0$, there is $C=C(n)>0$ so that
\begin{align}\label{equ-913-0}
\int_{\mathbb{R}^n}{|u_0(x)|^2\,\mathrm dx}\leq C\Big(1+\frac{1}{T}\Big)e^{Cr^2(1+\frac{1}{T})}\int_0^{T} \int_{B^c(x_0,\, r)}{|e^{-\i tH}u_0|^2\,\mathrm dx} \mathrm dt,\;\forall\, u_0\in L^2(\mathbb{R}^n).
\end{align}
\end{theorem}

\begin{theorem}\label{thm-1126-1}
 Let $E=B^c(0,\, r)\bigcap\{x\in\mathbb{R}^n:\, a\cdot x\ge 0\}$ with $r>0$ and $a\in \mathbb{S}^{n-1}$. Then the following assertions are equivalent:

 \noindent $(i)$ The set $E$ is an observable set at time $T>0$ for \eqref{equ1.1}, i.e.,
 there is $C=C(T, n)>0$ so that
\begin{align}\label{1120-1}
\int_{\mathbb{R}^n}{|u_0(x)|^2\,\mathrm dx}\leq  C\int_0^{T}\int_E{|e^{-\i tH}u_0|^2\,\mathrm dx}\,\mathrm dt,\; \forall\, u_0\in L^2(\mathbb{R}^n).
\end{align}

\noindent $(ii)$ It holds that $T>\frac{\pi}{2}$.
\end{theorem}

\begin{theorem}\label{thm-1126-3}
Suppose that $E\subset\mathbb{R}^n$ is an observable set at time $T>0$ for the equation \eqref{equ1.1}, i.e., there is $C=C(T,E)>0$ so that when  $u$ solves \eqref{equ1.1},
\begin{align}\label{equ1.26}
\int_{\mathbb{R}^n}{|u(0, x)|^2\,\mathrm dx}\leq C\int_0^{T}\int_{E}{|u(t, x)|^2\,\mathrm dx}\,\mathrm dt.
\end{align}
Then there is  $L>0$ and $c>0$ so that
\begin{align}\label{equ1.27}
\big|E\bigcap B(y,\, L\rho(y))\big|\geq ce^{-|y|^2}\;\;\mbox{for all}\;\; y\in\mathbb{R}^n,
\end{align}
where $\rho(y):=\max(1, |y|)$.
\end{theorem}

Several notes on Theorem \ref{thm-1126-2}-Theorem \ref{thm-1126-3} are as follows:
\begin{itemize}
\item [($\textbf{c}_1$)]  From   Theorem \ref{thm-1126-2} and  Theorem \ref{thm-1126-1}, we see that   for the  Hermite Schr\"{o}dinger equation in $\mathbb{R}^n$, different kinds of  $E$  lead to different types of observability: when $E$ is the complement of any closed ball, it is \emph{an observable set at any time} for \eqref{equ1.1}, while when $E$ is half of the complement of any closed ball, it is \emph{an observable set at  time $T>0$} for \eqref{equ1.1} if and only if  $T>\frac{\pi}{2}$. (These can be viewed as  supplements of Theorem \ref{thm-ob-for-HO}.)
     In this direction, { we mention that  Theorem \ref{thm-1126-1} where $n=1$ has been stated in \cite[Prop. 5.1]{DM} without proof.
     It also deserves  mentioning the} recent paper \cite{BS1} dealing with the time optimal observability for the two-dim Grushin Schr\"{o}dinger equation on the finite cylinder: $\Omega=(-1, 1)\times \mathbb{T}_y$.

\item [($\textbf{c}_2$)] The method to prove Theorem \ref{thm-1126-2} and
 Theorem \ref{thm-1126-1} differs from that to show Theorem \ref{thm-ob-for-HO}:
 the proof of Theorem \ref{thm-ob-for-HO} is based on an   abstract approach provided by  Proposition \ref{prop-re-to-ob}, while the proofs of Theorem \ref{thm-1126-2} and
 Theorem \ref{thm-1126-1} rely on  an observability inequality for \eqref{equ1.1} at two points in time, given by  Theorem \ref{thm-two-point-ob} which is based on Nazarov's uncertainty principle (see \cite{Jam})  and is independently interesting. With regard to the
  observability inequality at two points in time for Schr\"{o}dinger equations, we mention
  papers  \cite{WWZ} and \cite{HS}.

\item [($\textbf{c}_3$)] Theorem \ref{thm-1126-3} gives a necessary condition on \emph{observable sets at some time} for the $n$-dim Hermite Schr\"{o}dinger equation,
    while
Theorem \ref{thm-ob-for-HO} presents a necessary and sufficient condition on \emph{observable sets at some time} for the 1-dim Hermite Schr\"{o}dinger equation.

\item [($\textbf{c}_4$)] The  condition \eqref{equ1.27} is another kind of thick condition.
In the 1-dim case, it  is strictly weaker than the weakly thick condition \eqref{def-1-3} (see Remark \ref{remark5.2,11-18}).
 \end{itemize}

\subsection{Plan of the paper}

The rest of the paper is organized as follows:
 Section \ref{sec2-free-case} proves Theorem \ref{thm-ob-for-free};
 Section \ref{sec3} shows a sufficient condition (see Theorem  \ref{thm-ob-for-genral}) on \emph{the observable sets at some time}
 for Schr\"{o}dinger equations with more general potentials;
 Section \ref{sec4} gives the proof of Theorem \ref{thm-ob-for-HO}, as well as the comparison of thicknesses for two kinds of observable sets;
  Section \ref{two-point-hermite}  presents the proofs of  Theorem \ref{thm-1126-2}-Theorem \ref{thm-1126-3};  Appendix \ref{sec-app} is the last section where
 we give the proof of Lemma \ref{lem-append}, which is the key to show  the necessity of
  \emph{observable sets }  in Theorem \ref{thm-ob-for-HO}.

\section{Proof of Theorem \ref{thm-ob-for-free}}\label{sec2-free-case}

We start with introducing a resolvent condition on the observability for some evolution equation,
i.e., the next Proposition \ref{equ-ab-sch}, which is  another version of \cite[Theorem 5.1]{Mi}
(see also \cite{BZ}, \cite{Lau} {and \cite{RTTT}}). It will be used in  the proof of Theorem \ref{thm-ob-for-free}, as well as in the proof of Theorem \ref{thm-ob-for-HO}.
To state it, we consider the equation:
\begin{equation}\label{equ-ab-sch}
\i \partial_t u(t,x)= \mathscr{A} u(t,x), \;\; (t,x)\in \mathbb{R}^+\times \mathbb{R};\quad u(0,x)\in L^2(\R),
\end{equation}
where  $\mathscr{A}$ is a self-adjoint operator on $L^2(\R)$.
\emph{Observable sets at some time} for \eqref{equ-ab-sch} can be defined in the same manner as that in the definition ($\textbf{D}_1$) in Subsection \ref{sev def}.
\begin{proposition}\label{prop-re-to-ob}
Let  $E\subset \mathbb{R}$ be a measurable set. Then the following statements are equivalent:

\noindent $(i)$ The set  $E$ is an  observable set  at  some  time  for \eqref{equ-ab-sch}.

\noindent $(ii)$ There is $M>0$ and $m>0$ so that
  \begin{equation}\label{equ-resolvent-1}
  \|u\|^2_{L^2(\R)}\leq M\|(\mathscr{A}-\lambda)u\|^2_{L^2(\R)} + m\|u\|^2_{L^2(E)}\;\;\mbox{for all}\;\;
  u\in D(\mathscr{A})\;\;\mbox{and}\;\; \lambda \in \R.
  \end{equation}
   \end{proposition}
We now on the position to show Theorem \ref{thm-ob-for-free}.
\begin{proof}[\textbf{Proof of Theorem \ref{thm-ob-for-free}}]
Arbitrarily fix a measurable subset $E\subset \mathbb{R}$. We organize the proof by several steps.

\noindent{\it Step 1. We show that (i) of Theorem \ref{thm-ob-for-free} implies
\eqref{equ-resolvent-1} with $\mathscr{A}=-\partial_x^2$.}

Suppose that $E$ is thick. Arbitrarily fix $\lambda\in \mathbb{R}$ and $u\in H^2({\R})$.
In the case that $\lambda\geq 0$, we get from \cite[Proposition 1]{Gre} that for some
$m=m(E)>0$ and $C=C(E)>0$,
\begin{align}\label{equ-929-1}
 \|u\|^2_{L^2(\R)}\leq C\left(\|(-\partial_x^2-\lambda)u\|^2_{L^2(\R)} + m\|u\|^2_{L^2(E)} \right),
 \end{align}
which yields
\eqref{equ-resolvent-1} (with $\mathscr{A}=-\partial_x^2$) for the case when $\lambda\geq 0$.

In the case that $\lambda<0$, we have
$\|-\partial_x^2 u\|^2_{L^2(\R)}\leq \|(-\partial_x^2-\lambda)u\|^2_{L^2(\R)}$,
which, along with \eqref{equ-resolvent-1} where $\lambda=0$, leads to \eqref{equ-resolvent-1} (with $\mathscr{A}=-\partial_x^2$) for the case when $\lambda<0$.

\noindent{\it Step 2. We show that \eqref{equ-resolvent-1} (with $\mathscr{A}=-\partial_x^2$) implies  $(ii)$ of Theorem \ref{thm-ob-for-free}.}

Arbitrarily fix $c\in \mathbb{R}$.
Suppose that $E$ satisfies \eqref{equ-resolvent-1} (with $\mathscr{A}=-\partial_x^2$).
Then we have that for some $M>0$ and $m>0$,
\begin{eqnarray}\label{equ-92-2}
  \|u\|^2_{L^2(\R)}\leq M\|(-\partial_x^2-\lambda)u\|^2_{L^2(\R)} + m\|u\|^2_{L^2(E)}\;\;\mbox{for all}\;\;u\in D(\partial_x^2+c)
 \;\;\mbox{and}\;\;\lambda \in \R,
\end{eqnarray}
since $D(-\partial_x^2+c)=D(\partial_x^2)=H^2(\mathbb{R})$.
By \eqref{equ-92-2} and Proposition \ref{prop-re-to-ob} (with $\mathscr{A}=-\partial_x^2+c$), we find that $E$ is \emph{an observable set at some time} for
\eqref{equ-free-sch}. Hence, $(ii)$ of Theorem \ref{thm-ob-for-free} is true.

\noindent{\it Step 3. It is clear that (ii) of Theorem \ref{thm-ob-for-free} implies (iii) of Theorem \ref{thm-ob-for-free}.}

\noindent{\it Step 4. We show that (iii) of Theorem \ref{thm-ob-for-free} implies (i) of Theorem \ref{thm-ob-for-free}.}

We borrow some ideas from \cite{WWZZ} in this step. Recall that the  kernel of the Schr\"{o}dinger equation \eqref{equ-free-sch} with $c=0$ is
$$
K(t,x)=(4\pi \i t )^{-1/2}e^{-|x|^2/4\i t},\;\; t>0, x\in\mathbb R.
$$
Thus, given $u_0\in \mathscr{S}(\mathbb R)$ (the Schwartz class), the function defined by
\begin{eqnarray}\label{equ-92-3}
(t,x) \longrightarrow\int_{\mathbb{R}}K(t,x-y)u_0(y)\,\mathrm dy,\;\; (t,x)\in(0,\infty)\times\mathbb R,
\end{eqnarray}
is a solution to the equation \eqref{equ-free-sch} (where $c=0$), with the initial condition $u(0,x)=u_0(x)$, $x\in \mathbb{R}$. Arbitrarily fix $x_0\in \mathbb{R}$. By taking
$$
u_0(x)=(4\pi )^{-1/2}e^{-|x-x_0|^2/4},\;\;x\in\mathbb R,
$$
in \eqref{equ-92-3}, we get the following solution to the equation \eqref{equ-free-sch} (where $c=0$):
\begin{align}\label{special-solution}
v(t,x)=(4\pi (\i t+1))^{-\frac{1}{2}}e^{-\frac{|x-x_0|^2}{4(\i t+1)}}, \quad t\geq0, \;x\in\mathbb R.
\end{align}

We now suppose that $E$ satisfies (iii) of Theorem \ref{thm-ob-for-free}.
Then there is $T>0$ and  $C=C_{obs}(T,E)>0$ so that any solution $u$ to \eqref{equ-free-sch} (where $c=0$)
satisfies \eqref{ob-q}, from which, it follows  that
\begin{align}\label{equ-92-4}
\int_{\mathbb{R}}|v(0,x)|^2\,\mathrm dx&\leq C\int_0^{T}\int_{E}|v(t,x)|^2\,\mathrm dx\,\mathrm dt.
\end{align}
Besides, we have the following two observations: First,  a direct computation gives
\begin{align}\label{equ-92-5}
\int_{\mathbb{R}}|v(0,x)|^2\,\mathrm dx = \frac{1}{2\sqrt{2\pi}}.
\end{align}
Second,  for arbitrarily fixed $L>0$, we have
\begin{align}\label{equ-92-6}
&\int_0^{T}\int_{E}|v(t,x)|^2\,\mathrm dx\,\mathrm dt
\nonumber\\
& = \int_0^{T}\int_{E}\frac{1}{4\pi(1+t^2)}e^{-\frac{(x-x_0)^2}{2(1+t^2)}}\,\mathrm dx\,\mathrm dt \leq \int_0^{T}\int_{E}\frac{1}{2\pi}e^{-\frac{(x-x_0)^2}{2(1+T^2)}}\,\mathrm dx\,\mathrm dt \nonumber\\
&  = \frac{T}{2\pi}\left( \int_{E\bigcap[x_0-{L}/{2},x_0+{L}/{2}]} e^{-\frac{(x-x_0)^2}{2(1+T^2)}}\,\mathrm dx + \int_{E\bigcap[x_0-{L}/{2},x_0+{L}/{2}]^c} e^{-\frac{(x-x_0)^2}{2(1+T^2)}}\,\mathrm dx \right).
\end{align}
Since
$$
\int_{E\bigcap[x_0-{L}/{2},x_0+{L}/{2}]} e^{-\frac{(x-x_0)^2}{2(1+T^2)}}\,\mathrm dx \leq \left|E\bigcap[x_0-{L}/{2},x_0+{L}/{2}]\right|
$$
and
$$
\int_{E\bigcap[x_0-{L}/{2},x_0+{L}/{2}]^c} e^{-\frac{(x-x_0)^2}{2(1+T^2)}}\,\mathrm dx \leq  e^{-\frac{(L/2)^2}{4(1+T^2)}}\int_{E\bigcap[x_0-{L}/{2},x_0+{L}/{2}]^c} e^{-\frac{(x-x_0)^2}{4(1+T^2)}}\,\mathrm dx
\leq \sqrt{4\pi(1+T^2)}e^{-\frac{L^2}{16(1+T^2)}},
$$
we deduce from \eqref{equ-92-6} that
\begin{align}\label{equ-92-7}
&\int_0^{T}\int_{E}|v(t,x)|^2\,\mathrm dx\,\mathrm dt
\nonumber\\
&\leq \frac{T}{2\pi}\left( \big|E\bigcap[x_0-{L}/{2},x_0+{L}/{2}]\big|+  \sqrt{4\pi(1+T^2)}e^{-\frac{L^2}{16(1+T^2)}} \right).
\end{align}

Now, it follows from \eqref{equ-92-4}, \eqref{equ-92-5} and \eqref{equ-92-7} that
\begin{eqnarray*}\label{equ-92-8}
\frac{1}{2\sqrt{2\pi}} \leq \frac{CT}{2\pi}\left( \big|E\bigcap[x_0-{L}/{2},x_0+{L}/{2}]\big|+  \sqrt{4\pi(1+T^2)}e^{-\frac{L^2}{16(1+T^2)}} \right).
\end{eqnarray*}
In the above, by taking $L=L_0>0$ so that
$$
\frac{C{T}}{2\pi}\sqrt{4\pi(1+T^2)}e^{-\frac{L_0^2}{16(1+T^2)}} \leq \frac{1}{4\sqrt{2\pi}},
$$
 we find that
$$
\frac{1}{4\sqrt{2\pi}} \leq \frac{CT}{2\pi}  \big|E\bigcap[x_0-{L_0}/{2},x_0+{L_0}/{2}]\big|,
$$
from which, it follows  that
$$
 \big|E\bigcap[x_0-{L_0}/{2},x_0+{L_0}/{2}]\big| \geq \gamma L_0 \mbox{ with } \gamma = \frac{\sqrt{2\pi}}{4CTL_0}.
$$
Since $x_0$ was arbitrarily taken from $\R$, we obtain from the above that
$$
 \big|E\bigcap[x-{L_0}/{2},x+{L_0}/{2}]\big| \geq \gamma L_0\;\;\mbox{ for all }\;\; x\in \R.
$$
This implies that
$$
 \big|E\bigcap[x,x+L_0]\big| \geq \gamma L_0\;\;\mbox{ for all }\;\; x\in \R,
$$
i.e.,  $E$ is a thick set.

Thus, we end the proof of Theorem \ref{thm-ob-for-free}.
\end{proof}

\section{Observable sets for more general  potentials}\label{sec3}

In this section, we will give a sufficient condition on  \emph{observable sets at some time} for the
Schr\"{o}dinger equation:
\begin{align}\label{equ3.8}
\i \partial_t u(t,x)=({-\partial_x^2+V(x)})u(t,x), \;\; (t,x)\in \mathbb{R}^+\times \mathbb{R}; \quad u(0,x)\in L^2(\R),
\end{align}
where the potential $V$ satisfies the following condition:

\textbf{Condition (H).} The real-valued function $V$  belongs to the space $C^3(\mathbb{R})$ and there is
  $c\ge1$ so that

(\romannumeral1) for some compact interval $K\Subset \mathbb{R}$, $V''(x)>0$ and  $x V'(x)\ge 2cV(x)>0$, when $x\in\mathbb{R}\setminus K$;

(\romannumeral2) when $j=1, 2, 3$, $V^{(j)}(x)=O(x^{-j}|x|^{2c})$ as $x\rightarrow\infty$.

{\it Here and in what follows, given two functions $f$ and $g$, by $f(x)=O(g(x))$ as $x\rightarrow\infty$, we mean that
there is $C>0$  and $M>0$ so that $|f(x)|\leq C|g(x)|$, when $|x|\geq M$.}

Several notes on \textbf{Condition (H)} are given in order.

\begin{itemize}
 \item [($\textbf{d}_1$)]
 \textbf{Condition (H)} is a variant of the condition given in  \cite{Ya}, where the smoothness of the fundamental solution of Schr\"{o}dinger equations with similar perturbations was studied. Typical examples of potentials satisfying this condition are as:
$$
V_1(x)=C(1+|x|^2)^c,\; x\in \mathbb{R},\;\;\mbox{with}\;\;C>0,\,\,\,c\ge1;
$$
and
$$
V_2(x)=\sum_{j=0}^{2m}a_j x^j,\; x\in \mathbb{R},\;\;\mbox{with}\;\;a_{2m}>0,\,  a_j\in \R, \,\, m\in \mathbb{N}^+.
$$
  \item [($\textbf{d}_2$)] By  \textbf{Condition (H)}, we have that for some $x_0>0$,
  \begin{eqnarray}\label{3.29,11-5}
V'(x)>0,\;\mbox{when}\;x\geq x_0;
\end{eqnarray}
  we also have two constants $D>D'>0$ so that
\begin{align}\label{equ3.2.1}
D'x^{2c}\leq V(x)\leq Dx^{2c},\;\mbox{when}\;|x|\geq x_0.
\end{align}
The later shows that $V(x)\rightarrow +\infty$ as $|x|$ goes to $+\infty$.
Hence, \textbf{Condition (H)} implies the following weaker condition:
\begin{align}\label{equ-poten}
V \mbox{ is real-valued, locally bounded and } V(x)\rightarrow +\infty,\,\,\,\text{as}\,\,|x|\rightarrow\infty.
\end{align}

 \item [($\textbf{d}_3$)]
Suppose that  $V$ satisfies \eqref{equ-poten}. Then, according to \cite[Theorem 1.1, p.50]{BS}, the operator
$H=-\partial_x^2+V$ has the properties: it is  essentially self-adjoint (i.e., its closure is self-adjoint);
  its resolvent is compact. Thus, we have $\sigma(H)=\{\lambda_k\}_{k=1}^\infty$, with
  \begin{eqnarray}\label{3.4-10-15}
 \lambda_1<\lambda_2<\cdots<\lambda_k\rightarrow +\infty,
 \end{eqnarray}
 where $\lambda_k$, $k\in \mathbb{N}^+$, are all eigenvalues of $H$. We further have that  each $\lambda_k$ is simple. (See Proposition \ref{prop-H}.)

 \item [($\textbf{d}_4$)] Throughout this section, we write $\{\lambda_k\}_{k=1}^\infty$, with \eqref{3.4-10-15},
 for the family of all eigenvalues of $H$, and $\{\varphi_k\}_{k=1}^\infty$ for the family of corresponding normalized eigenfunctions.

\end{itemize}
The main theorem of this section is as:

\begin{theorem}\label{thm-ob-for-genral}
Suppose that   \textbf{Condition (H)} holds. If $E\subset \R$ is a measurable set satisfying
\begin{align}\label{equ-930-4}
\varliminf_{x\rightarrow +\infty} \frac{|E\bigcap [0,x]|}{x}>0,
\end{align}
then  $E$ is an observable set at some time for the equation \eqref{equ3.8}.
\end{theorem}

\subsection{Properties of the operator $H$ with the weak condition \eqref{equ-poten}}\label{sec3.1-abstract-cond}

In this subsection, we will study some properties of eigenvalues and eigenfunctions of $H$ under the assumption
\eqref{equ-poten}.
The conclusions (i) and (ii) of the following Proposition \ref{prop-H} are given by  \cite[Corollary, p.64]{BS} and \cite[Proposition 3.3, p.65]{BS}, respectively.

\begin{proposition}\label{prop-H}
Assume that \eqref{equ-poten} holds. Then the following statements are true:

\noindent $(i)$ Every eigenfunction $\varphi_k$ of $H$ has at most  finite zero points.

\noindent $(ii)$ Every eigenvalue $\lambda_k$ of $H$ is simple, i.e., each eigenspace has dimension one.

\end{proposition}

Notice that under \eqref{equ-poten}, $H$ is essentially self-adjoint. Thus we have the following facts:
First, we can put the equation:
\begin{equation}\label{equ-929-5}
\i \partial_t u(t,x)=({-\partial_x^2+V(x)})u(t,x),\;\;t\in \mathbb{R}^+, x\in \mathbb{R}; \quad u(0,x)\in L^2(\R)
\end{equation}
into the framework of \eqref{equ-ab-sch}. Second, \emph{the observable sets at some time} for \eqref{equ-929-5}  can be defined
 in the same manner as that in the definition ($\textbf{D}_1$) in Subsection \ref{sev def}.
Third,
 $-\i H$ generates a  unitary group $e^{-\i tH}$  in $L^2(\R)$.
 Thus, the solution to equation \eqref{equ-929-5} is as:
 $u(t,\cdot) =  e^{-\i tH}u(0,\cdot)$ for each $t\in \R^+$.

The next proposition gives connections among  \emph{observable sets},  eigenfunctions and eigenvalues of $H$.

\begin{proposition}\label{pro-ob-general}
Suppose that \eqref{equ-poten} is true. Further assume that
eigenvalues of $H$ satisfy that
for some  $\varepsilon_0>0$ (independent of $k$),
\begin{align}\label{equ-929-11}
 \lambda_{k+1}-\lambda_k\ge \varepsilon_0>0 \;\; \mbox{for all}\;\; k\geq 1.
\end{align}
Then for any measurable set $E\subset\mathbb{R}$, the following statements  are equivalent:

\noindent $(i)$ The set $E$ is an observable set at some time for \eqref{equ-929-5}.

\noindent $(ii)$ The {$L^2$-}normalized eigenfunctions of $H$
 satisfy that for some $C>0$ (independent of $k$),
\begin{align}\label{equ-929-9}
\int_{E}|\varphi_k(x)|^2\d x \geq C>0\;\;\mbox{ for all }\;\; k\geq 1.
\end{align}
In addition, if  $(ii)$ holds and $H$ satisfies the following stronger spectral gap condition:
\begin{align}\label{equ-929-11-01}
 \lambda_{k+1}-\lambda_k\rightarrow \infty, \quad \mbox{as} \quad k\rightarrow \infty,
\end{align}
then $E$ is an observable  set at any time  for \eqref{equ-929-5}.
\end{proposition}

\begin{proof}
First, from the proof of \cite[Theorem 3.1, p.57]{BS}, we see that  $H$ has a compact resolvent. Then according to  \cite[Theorem 1.3, p.195]{RTTT}, we have the fact:
$E$ is  \emph{an observable set at some time} for \eqref{equ-929-5}
 if and only if there exists $ \varepsilon> 0$ and $C>0$ such that
\begin{eqnarray}\label{equ-929-12}
\int_{E}|\varphi(x)|^2\d x \geq C\int_{\R}|\varphi(x)|^2\d x\;\;\mbox{ for all }\;\; \varphi \in \bigcup_{\lambda\in \R}\mbox{span} \Big\{\varphi_m\;:\; m\in J_\varepsilon(\lambda)\Big\},
\end{eqnarray}
where
\begin{eqnarray}\label{3.9,10-14}
J_\varepsilon(\lambda) := \{m\in \mathbb{N}^+ \;:\; |\lambda_m-\lambda|<\varepsilon\}.
 \end{eqnarray}
 From this, we in particular  have what follows:
 \begin{eqnarray}\label{3.10,10-14}
 \mbox{The conclusion } (i) \mbox{ of Proposition \ref{pro-ob-general}} \Longleftrightarrow \eqref{equ-929-12}\;\;\mbox{with}\;\;\varepsilon=\frac{\varepsilon_0}{2}.
 \end{eqnarray}
 (Here, $\varepsilon_0$ is given by \eqref{equ-929-11}).

 Next,  by \eqref{equ-929-11}, \eqref{3.9,10-14} (where $\varepsilon=\varepsilon_0/2$) and $(ii)$ of Proposition \ref{prop-H}, one can directly check that
 \begin{eqnarray*}
 \bigcup_{\lambda\in \R}\mbox{span} \Big\{\varphi_m\;:\; m\in J_\varepsilon(\lambda)\Big\}
 =\bigcup_{k\in \mathbb{N}^+} \{a\varphi_k\;:\; a\in \mathbb{C} \},\;\;\mbox{with}\; \varepsilon=\frac{\varepsilon_0}{2}.
 \end{eqnarray*}
 This yields that
 \begin{eqnarray}\label{3.11,10.14}
  \eqref{equ-929-12}\;\;\mbox{with}\;\;\varepsilon=\frac{\varepsilon_0}{2} \Leftrightarrow \int_{E}|\varphi_k(x)|^2\d x \geq C\int_{\R}|\varphi_k(x)|^2\d x\;\mbox{for all}\;k\in \mathbb{N}^+.
 \end{eqnarray}

 Now,  since   $\int_{\R}|\varphi_k(x)|^2\d x=1$ for all $k\in \mathbb{N}^+$, it follows from
 \eqref{3.10,10-14} and \eqref{3.11,10.14} that $(i)$ $\Leftrightarrow$ $(ii)$.

 Finally, if \eqref{equ-929-11-01} holds, then we can apply \cite[Corolory 6.9.6]{TW}
 to see directly that $E$ is \emph{an observable  set at any time}  for \eqref{equ-929-5}.
       This ends the proof of Proposition \ref{prop-H}.
\end{proof}

\begin{lemma}[Lower bound for low frequency]\label{lem-low-f}
Suppose that \eqref{equ-poten} is true.   Then for
each subset $E\subset\mathbb{R}$ of positive measure and each  $\ell\in \mathbb{N}^+$, there exists $C=C(E,\ell)>0$ so that
\begin{eqnarray}\label{3.13-10-15}
\int_{E}|\varphi_k(x)|^2\d x \geq C, \quad \mbox{when}\; 1\leq k\leq \ell.
\end{eqnarray}
\end{lemma}
\begin{proof}
Arbitrarily fix a subset $E\subset\mathbb{R}$ of positive measure and   $\ell\in \mathbb{N}^+$.
First, by $(i)$ of Proposition \ref{prop-H}, we can easily see that for each $1\leq k\leq\ell$, there is $C_k = C(k,E)>0$
so that
\begin{align}\label{equ-930-1}
\int_{E}|\varphi_k(x)|^2\d x \geq C_k.
\end{align}
Next, by setting $C:=\min_{1\leq k\leq \ell}C_k>0$, we get \eqref{3.13-10-15} from \eqref{equ-930-1} at once.
This ends the proof of Lemma \ref{lem-low-f}.
\end{proof}

\subsection{Properties of the operator $H$ with \textbf{Condition (H)}}

In this subsection, we will study some properties of eigenvalues and eigenfunctions of $H$ under \textbf{Condition (H)}.   In particular,  we shall give  a uniform lower bound of high frequency eigenfunctions $\varphi_k$ for all $k>\ell$.
The  ideas in  \cite{Fe,Ti} (where the WKB method was applied to study  the  asymptotic behaviors of eigenfunctions of $H$) will be used.

We start with several facts. Fact One:
  Each $\varphi_k$ satisfies
 \begin{align}\label{equ3.1}
-\varphi_k''(x)+V(x)\varphi_k(x)=\lambda_k \varphi_k(x),\; x\in \mathbb{R}.
\end{align}

Fact Two: By \textbf{Condition (H)} and by \eqref{3.4-10-15}, there is $\bar k_0 \in \mathbb{N}^+$ so that $\lambda_k\geq 1$, when $k\geq \bar k_0$ and so that
\begin{eqnarray}\label{3.20,11-7}
0\in \Omega_k,\;\;\mbox{when}\;\;k\geq \bar k_0,
\end{eqnarray}
where
\begin{align}\label{equ3.2}
\Omega_{k}:=\{x\in \mathbb{R},\,\,\,V(x)\le \lambda_k/2\},\;\;k \geq \bar k_0.
\end{align}
Moreover, we find from   \eqref{equ3.2.1} that $\Omega_{k}$ is a bounded set and that
for some $C>0$ (independent of $k$),
\begin{align}\label{equ3.3}
|\Omega_k|\leq C\lambda_k^{\frac{1}{2c}}\;\;\mbox{for all}\;\; k\geq \bar k_0.
\end{align}

Fact Three:
Recall the following Liouville transform (see e.g. \cite[p. 119]{Ti}):
\begin{equation}\label{equ3.4}
\begin{cases}
y=S(x)=\int_0^x{\sqrt{\lambda_k-V(s)}\,\mathrm ds},\;\; x\in\Omega_k \\[4pt]
w=w(y)=\Big(\lambda_k-V(S^{-1}(y))\Big)^{\frac14} \varphi_k\Big(S^{-1}(y)\Big), \;\; y\in S(\Omega_k).
\end{cases}
\end{equation}
By \eqref{equ3.4}, we see that
\begin{eqnarray}\label{3.22.11-7}
S'(x)>0\;\mbox{over}\; \Omega_k;\;\;\;\;S^{-1}(\cdot)\;\;\mbox{exists over}\; S(\Omega_k).
\end{eqnarray}
By \eqref{equ3.4} and \eqref{3.20,11-7}, we obtain
\begin{eqnarray}\label{3.23.11-7}
0\in S(\Omega_k),\;\;\mbox{as}\;\;k\geq\bar k_0.
\end{eqnarray}
For each $k \geq \bar k_0$, by making the above Liouville transform
to \eqref{equ3.1}, which is restricted over $\Omega_k$, we obtain
\begin{align}\label{equ3.1'}
\frac{d^2w(y)}{dy^2}+w(y)+q(y)w(y)=0,\; y\in S(\Omega_k),
\end{align}
where
$$
q(y)=\frac{V''(x)}{4(\lambda_k-V(x))^2}+\frac{5(V'(x))^2}{16(\lambda_k-V(x))^3},\;\;\mbox{with}\;\;x=S^{-1}(y)\in \Omega_k.
$$

Fact Four: The function $w$ (given by \eqref{equ3.4}) depends on $k$.
By \eqref{equ3.4} and \eqref{3.23.11-7}, we see that
when $k\geq \bar k_0$,
 $0$ is in the domain of $w$. Hence, $w(0)$
and $w'(0)$ make sense.

Fact Five: The  next Lemma \ref{lem3.1} is quoted from \cite{Ya} (see \cite[Lemma 3.1\&3.2]{Ya}) and will  play an important role in our studies. In the proof of Lemma \ref{lem3.1}, \eqref{equ3.1'} was used.

\begin{lemma}\label{lem3.1}
Suppose that \textbf{Condition (H)} holds for some $c\geq 1$.
Let $\Omega_k$ be given by  \eqref{equ3.2}. Let $S$ and $w$ be given by  \eqref{equ3.4}. Then there exists  $C>0$ and $\widetilde{k}_0\in \mathbb{N}^+$ (with $\widetilde{k}_0\geq \bar k_0$ which is given by \eqref{3.20,11-7}) so that when $k\geq \widetilde{k}_0$,
\begin{eqnarray}\label{3.21,11-6}
\varphi_k(x)=(\lambda_k-V(x))^{-\frac14}\cdot \Re(C_{\lambda_k}e^{\i S(x)})+R_k(x)\;\;\mbox{for each}\; x\in \Omega_k,
\end{eqnarray}
where $R_k$ is a function  with the estimate:
\begin{equation}\label{equ3.5}
|R_k(x)|
\leq C(\lambda_k-V(x))^{-\frac14}\cdot|C_{\lambda_k}|\cdot\lambda_k^{-\frac12}\;\;\mbox{for each}\; x\in \Omega_k,
\end{equation}
with
\begin{eqnarray}\label{3.22,11.4}
C_{\lambda_k}=w(0)-\i w'(0).
\end{eqnarray}
Moreover, there is $C'>0$ so that
\begin{align}\label{equ3.6}
|C_{\lambda_k}|\ge C' \lambda_k^{\frac14-\frac{1}{4c}}\;\;\mbox{for all}\;\; k\geq \widetilde{k}_0.
\end{align}
\end{lemma}

Next, we will give an upper bound for the family  $\{|C_{\lambda_k}|\}$. {\it That upper bound shows
that  the lower bound in \eqref{equ3.6} is sharp, as a byproduct.}

\begin{lemma}\label{lem-930-C}
Suppose that \textbf{Condition (H)} (with $c\geq 1$) holds.
Let $C_{\lambda_k}$ be given by \eqref{3.22,11.4}
in Lemma \ref{lem3.1}.  Then there is  $\hat k_0\in \mathbb{N}^+$ and $C>0$ so that
\begin{align}\label{equ3.15}
|C_{\lambda_k}|\le C \lambda_k^{\frac14-\frac{1}{4c}}\;\;\mbox{for all}\;\;k\geq \hat k_0.
\end{align}
\end{lemma}
\begin{proof}
By \textbf{Condition (H)}, we have the notes ($\textbf{d}_2$) (see  \eqref{3.29,11-5} and \eqref{equ3.2.1})
and ($\textbf{d}_3$) (see \eqref{3.4-10-15}).
Let $c$ and $x_0$ be given by \textbf{Condition (H)} and the note ($\textbf{d}_2$) respectively.
According to \eqref{3.4-10-15}, there is  $\hat k_0\in \mathbb{N}^+$, with
$\hat k_0\geq  \widetilde{k}_0$
(where  $\widetilde{k}_0$ is given by  Lemma \ref{lem3.1})
 so that when
$  k\geq \hat k_0$,
\begin{align}\label{equ-1020-1}
\lambda_k\geq \max\left\{1,\, 2D(2(x_0+1))^{2c},\, \left(6\sqrt{2}\pi\right)^{2c/(c+1)}(2D)^{1/(c+1)},\, \left(\frac{4C}{\sqrt{3}} \right)^2\right\},
\end{align}
where  $C$ and $D$ are given by
 \eqref{equ3.5} and \eqref{equ3.2.1} respectively.
Arbitrarily fix $ k\geq \hat k_0$. We divide the rest of the proof into several steps.

\noindent \emph{Step 1. Define the following interval:
\begin{align}\label{equ-1025-3}
I_k:=[x_k/2, x_k],\;\;\mbox{with}\;\;x_k:=\alpha\lambda_k^{\frac{1}{2c}},\,\,\,\,   \alpha := \left(\frac{1}{2D}\right)^{\frac{1}{2c}}.
\end{align}
We claim that
\begin{align}\label{equ-1025-2.5}
I_k\subset\Omega_k\bigcap[x_0+1,\infty).
\end{align}}

We first show that $I_k\subset [x_0+1,\infty)$. Indeed, by \eqref{equ-1020-1}, we have $\lambda_k\geq   2D(2(x_0+1))^{2c}$, which, along with \eqref{equ-1025-3}, yields $x_k\geq  2(x_0+1)$. This, together with \eqref{equ-1025-3}, leads to $I_k\subset [x_0+1,\infty)$.

We next show that $I_k\subset \Omega_k$. In fact,
since $I_k\subset [x_0,\infty)$, it follows from \eqref{3.29,11-5} that
\begin{align}\label{equ-1025-4}
V'(x)>0\;\;\mbox{ for all }\; x\in I_k.
\end{align}
Meanwhile,   by the definitions of $x_k$ and $\alpha$ (see \eqref{equ-1025-3}),
and by \eqref{equ3.2.1}, we  find
\begin{align}\label{equ-1025-5}
  V(x_k) = V(\alpha\lambda_k^{\frac{1}{2c}}) \leq \lambda_k/2.
\end{align}
Combining \eqref{equ-1025-4} and \eqref{equ-1025-5} gives that
$$
V(x)\leq \lambda_k/2\;\;\mbox{ for all }\; x\in I_k.
$$
This, along with  \eqref{equ3.2}, leads to $I_k\subset \Omega_k$.
Hence, \eqref{equ-1025-2.5} has been proved.

\noindent
\emph{Step 2.   Define the following set:
\begin{align}\label{equ-1020-2}
J:= \{ j\in \mathbb{N}:   S(x_k/2)+\pi\leq j\pi+\theta_0\leq S(x_k)-\pi\},
\end{align}
where $S(\cdot)$ is given by \eqref{equ3.4} and
$\theta_0$ is defined as:
\begin{align}\label{equ-1025-7}
\theta_0:=\arctan{w'(0)/w(0)},\;\mbox{when}\;w(0)\neq 0;\;\;\theta_0:=\pi/2,\;\mbox{when}\;w(0)=0,
\end{align}
where $w(\cdot)$ is given by \eqref{equ3.4}.
We claim
 \begin{align}\label{equ-1020-3}
\sharp J \geq \frac{\alpha}{6\pi\sqrt{2}}\lambda_k^{\frac{1}{2}+\frac{1}{2c}}\geq 1.
\end{align}
Here  and below, $\sharp J$ denotes the cardinality of $J$.}

Two facts deserve to be mentioned: First,
\eqref{equ-1020-3} holds  for all $\theta_0\in [-\pi/2,\pi/2]$ and that we choose
it satisfying \eqref{equ-1025-7} for later use.
Second, when \eqref{equ-1020-3} is proved, we have that $J\neq\emptyset$.

To prove \eqref{equ-1020-3}, we see from \eqref{equ-1025-2.5}, \eqref{equ3.2.1} and  \eqref{equ3.2} that
\begin{align}\label{equ-1026-3}
0\leq V(x)\leq \lambda_k/2, \quad x\in I_k.
\end{align}
Then, by the definition of $S(x)$ (see \eqref{equ3.4}) and by \eqref{equ-1026-3}, we have
$$
 S(x_k)-S(x_k/2) = \int_{x_k/2}^{x_k}\sqrt{\lambda_k-V(x)}\, \d x \geq \frac{x_k\sqrt{\lambda_k}}{2\sqrt{2}}.
$$
Since $\lambda_k\geq \left(6\sqrt{2}\pi\right)^{2c/(c+1)}(2D)^{1/(c+1)}$
(see \eqref{equ-1020-1})
and because $\theta_0\in (-\pi/2,\pi/2]$ (see \eqref{equ-1025-7}), the above,  along with \eqref{equ-1020-2}
and the definitions of $x_k$ and $\alpha$ (see \eqref{equ-1025-3}), shows
$$
\sharp J \geq \frac{ S(x_k)-S(x_k/2)-2\pi}{\pi}  \geq \frac{x_k\sqrt{\lambda_k}}{2\pi\sqrt{2}} -2=\frac{\alpha \lambda_k^{\frac{1}{2}+\frac{1}{2c}}}{2\pi\sqrt{2}}-2   \geq\frac{\alpha}{ 6\pi \sqrt{2}}\lambda_k^{\frac{1}{2}+\frac{1}{2c}}\geq 1,
$$
which leads to \eqref{equ-1020-3}.

\noindent \emph{Step 3. Define, for each $j\in J$, the following set:
\begin{align}\label{equ-1020-6}
E_{k, j}:=[x_{k, j}-\mu x_{k, j}^{-c}, x_{k, j}],\;\;\mbox{with}\; \mu:=\frac{\pi}{6\cdot2^c\sqrt{D}},
\end{align}
where   $x_{k, j}$   is the unique solution to the equation:
\begin{align}\label{equ3.11}
S(x)=j\pi+\theta_0, \quad x\in I_k.
\end{align}
 We claim
\begin{align}\label{equ-114-3}
E_{k,j}\subset I_k, \quad  j\in J;
\end{align}
\begin{align}\label{equ3.17}
E_{k, j}\bigcap E_{k, j'}=\emptyset\;\;\mbox{for all} \; j,j'\in J \;\;\mbox{with}\;\; j\ne j';
\end{align}
\begin{align}\label{equ3.14}
|\cos{(S(x)-\theta_0)}|\geq \sqrt{3}/2\;\;\mbox{for all}\;\; x\in E_{k,j}\;\;\mbox{and}\;\; j\in J.
\end{align} }

First of all, by \eqref{3.22.11-7} and \eqref{equ-1025-2.5}, we infer that the function $S(\cdot)$ is strictly increasing over  $I_k$. This fact and the definition \eqref{equ-1020-2} imply that   the equation \eqref{equ3.11} has a unique solution $x_{k,j}$, which satisfies that $x_k/2\leq x_{k,j}\leq x_k$ for all $j\in J$.

We  claim that for each $j\in J$,
\begin{align}\label{equ-114-2}
 0\leq V(x)\leq \lambda_k/2,\;\; \mbox{when}\; x\in E_{k,j}.
\end{align}
To this end, we arbitrarily fix $j\in J$. Since $x_k/2\leq x_{k,j}$; $x_k/2\geq x_0+1$ (see \eqref{equ-1025-2.5});  $\mu(x_k/2)^{-c}\leq 1$ (see \eqref{equ-1020-1} and \eqref{equ-1020-6}),
 we have
 \begin{align}\label{equ-114-1}
x_{k, j}-\mu x_{k, j}^{-c} \geq (x_k/2) -\mu(x_k/2)^{-c}\geq x_0.
\end{align}
Since  $x_{k,j}\leq x_k$, it follows by  \eqref{equ-114-1} that
\begin{eqnarray}\label{3.43,11-5}
E_{k, j}\subset [x_0,  x_k].
\end{eqnarray}
Notice that $V'(x)>0$ and $V(x)>0$ over $[x_0,+\infty)$ (see \eqref{3.29,11-5}
and \eqref{equ3.2.1}). From these, \eqref{3.43,11-5} and \eqref{equ-1025-5}, we are led to
\eqref{equ-114-2}.

 We now show \eqref{equ3.14}.
  Since $x_{k,j}$ satisfies \eqref{equ3.11}, we have that when $x\in E_{k,j}$,
\begin{align}\label{equ-1020-4}
|S(x)-\theta_0-j\pi|&=|S(x)-S(x_{k, j})|\leq|x-x_{k, j}|\cdot\sup_{x\in E_{k,j}}|V'(x)|\nonumber\\
&=  |x-x_{k, j}|\cdot\sup_{x\in E_{k,j}}\sqrt{\lambda_k-V(x)} \nonumber\\
&\leq \mu x_{k, j}^{-c}\sqrt{\frac{\lambda_k}{2}}\leq\frac{\mu}{\sqrt{2}}2^c\sqrt{2D} =\frac{\pi}{6}.
\end{align}
(On the last line of \eqref{equ-1020-4}, we used \eqref{equ-114-2} and the fact  $x_{k,j}\geq x_k/2= 2^{-1}(2D)^{-1/(2c)}\lambda_k^{1/(2c)}$, as well as the definition of $\mu$ in \eqref{equ-1020-6}.)
From \eqref{equ-1020-4}, we are led to \eqref{equ3.14} at once.

We next show \eqref{equ-114-3}. Indeed,  by the monotonicity of $S^{-1}$ on $I_k$ (see \eqref{3.22.11-7} and \eqref{equ-1025-2.5})
and  by \eqref{equ-1020-2}, we see
\begin{align}\label{equ-114-4}
x_k/2\leq S^{-1}(j\pi+\theta_0-\frac{\pi}{6}) \leq S^{-1}( j\pi+\theta_0+\frac{\pi}{6}) \leq x_k  \;\;\mbox{for all}\; j\in J.
\end{align}
Meanwhile, by \eqref{equ-1020-4}, we have
\begin{eqnarray}\label{3.46,11-5}
E_{k,j} \subset S^{-1}([j\pi+\theta_0-\frac{\pi}{6}, j\pi+\theta_0+\frac{\pi}{6}])\;\;\mbox{for all}\; j\in J.
\end{eqnarray}
From \eqref{3.46,11-5}  and \eqref{equ-114-4}, we obtain  $E_{k,j}\subset [x_k/2, x_k]$, i.e.,  \eqref{equ-114-3} holds.

We finally show \eqref{equ3.17}. Indeed,  it follows from \eqref{3.46,11-5} that
  for all $j,j'\in J$ with  $j\ne j'$,
\begin{eqnarray}\label{equ-1025-6}
E_{k, j}\bigcap E_{k, j'}\subset S^{-1}\left(K_{k, j}\right)\bigcap S^{-1}\left(K_{k, j'}\right),
\end{eqnarray}
where
\begin{eqnarray*}
K_{k, j}:=\left[j\pi+\theta_0-\frac{\pi}{6}, j\pi+\theta_0+\frac{\pi}{6}\right];
\end{eqnarray*}
\begin{eqnarray*}
K_{k, j'}:=\left[j'\pi+\theta_0-\frac{\pi}{6}, j'\pi+\theta_0+\frac{\pi}{6}\right].
\end{eqnarray*}
Since $S^{-1}$ is strictly monotonic on $I_k$ (see \eqref{3.22.11-7} and \eqref{equ-1025-2.5}),  the set on the right hand side of \eqref{equ-1025-6} is empty.
This leads to \eqref{equ3.17}.

\noindent
\emph{Step 4. Define the following subset:
\begin{align}\label{equ3-101-3}
G_k: = \left\{ x\in I_k:\,\,  |\cos{(S(x)-\theta_0)}|\geq \frac{\sqrt{3}}{2}  \right\}.
\end{align}
We claim
\begin{align}\label{equ-1020-45}
|\varphi_k(x)| \geq \frac{\sqrt{3}}{4}\left(\frac{\lambda_k}{2}\right)^{-\frac{1}{4}}|C_{\lambda_k}|, \quad x\in G_k,
\end{align}
and
\begin{align}\label{equ-1020-5}
G_k\supset\bigcup_{j\in J}E_{k, j}.
\end{align} }

We start with proving \eqref{equ-1020-45}. Because $k\geq \hat k_0\geq \widetilde{k}_0$,
 the results in Lemma \ref{lem3.1} are valid for  $k$.
Thus, we can use   \eqref{3.21,11-6},
 \eqref{equ3.5},  \eqref{3.22,11.4} and \eqref{equ-1025-7} to  get
\begin{align}\label{equ3-1003-5}
\Re (C_{\lambda_k}e^{iS(x)})=w(0)\cos{S(x)}+w'(0)\sin{S(x)}=|C_{\lambda_k}|\cos{(S(x)-\theta_0)}, \;x\in \Omega_k.
\end{align}
 From \eqref{3.21,11-6}, \eqref{equ3.5}, \eqref{equ3-1003-5}, \eqref{equ3-101-3} and \eqref{equ-1026-3}, we see that
\begin{align*}
|\varphi_k(x)|&\geq (\lambda_k-V(x))^{-\frac{1}{4}}\Big( |\Re (C_{\lambda_k}e^{iS(x)})| - C|C_{\lambda_k}|\lambda_k^{-\frac{1}{2}}\big)\\
& \geq \left(\frac{\lambda_k}{2}\right)^{-\frac{1}{4}}|C_{\lambda_k}|(\frac{\sqrt{3}}{2}-C\lambda_k^{-\frac{1}{2}}) \geq \frac{\sqrt{3}}{4}\left(\frac{\lambda_k}{2}\right)^{-\frac{1}{4}}|C_{\lambda_k}|, \;\; \mbox{when}\; x\in G_k.
\end{align*}
(On the last step above, we used the fact that $\lambda_k\geq (\frac{4C}{\sqrt{3}})^2$ which follows from \eqref{equ-1020-1}.)
 This leads to \eqref{equ-1020-45}.

 Now \eqref{equ-1020-5} follows from
 \eqref{equ-114-3}, \eqref{equ3.14} and \eqref{equ3-101-3} directly.

\noindent\emph{Step 5. We complete the proof.}

We have
\begin{align}\label{equ3.18}
1&\ge\int_{G_k}{|\varphi_k|^2\, \d x}\nonumber\\
 (\mbox{by} \eqref{equ-1020-45}) &\geq \int_{G_k} \frac{3}{16} \left(\frac{\lambda_k}{2}\right)^{-\frac{1}{2}}|C_{\lambda_k}|^2\d x \nonumber\\
(\mbox{ by } \eqref{equ-1020-5} \mbox{ and } \eqref{equ3.17})&\geq \sum_{j\in J} \int_{E_{k,j}} \frac{3}{16} \left(\frac{\lambda_k}{2}\right)^{-\frac{1}{2}}|C_{\lambda_k}|^2\, \d x\nonumber\\
&  =  \frac{3}{16} \left(\frac{\lambda_k}{2}\right)^{-\frac{1}{2}}|C_{\lambda_k}|^2 \cdot \sharp J \cdot |E_{k,j}| \nonumber\\
(\mbox{ by } \eqref{equ-1020-3} \mbox{ and } \eqref{equ-1020-6})&\geq  \frac{3\alpha \mu}{96\pi}|C_{\lambda_k}|^2 \lambda_k^{\frac{1}{2c}}x_{k,j}^{-c}\nonumber\\
&\geq \frac{3 \mu}{96\pi}  \alpha^{1-c}\lambda_k^{-\frac{1}{2}+\frac{1}{2c}}|C_{\lambda_k}|^2.
\end{align}
(On the last inequality in \eqref{equ3.18}, we   used the fact that $x_{k, j}\leq x_k$ and the definition of $x_k$ in \eqref{equ-1025-3}.) By  \eqref{equ3.18}, we see
$$
|C_{\lambda_k}|\leq C(D,C,c)\lambda_k^{\frac{1}{4}-\frac{1}{4c}}.
$$
Since the above holds for any  $k\geq \hat k_0$, we obtain  \eqref{equ3.15}.
This ends the proof of Lemma \ref{lem-930-C}.
\end{proof}

The inequality \eqref{equ-1020-45} gives a point-wise lower bound for each eigenfunction with the high frequency. Some ideas used in its proof can be borrowed
to build up the following uniform lower bound for eigenfunctions (with the high frequency) over some kind of thick set:
\begin{lemma}[Lower bound for high frequency]\label{lem-low-h}
Suppose that \textbf{Condition (H)}  holds. If $E\subset \R$ is a measurable set satisfying
\begin{align}\label{equ3-1003-1}
\varliminf_{x\rightarrow +\infty} \frac{|E\bigcap [0,x]|}{x}>0,
\end{align}
then there exists  $k_0\in \mathbb{N}^+$ and $C>0$ so that
\begin{align}\label{equ3-1003-2}
\int_E|\varphi_k(x)|^2\d x \geq C\;\;\mbox{ for all }\;\; k\geq k_0.
\end{align}
\end{lemma}
\begin{proof}
Suppose that $E$ satisfies \eqref{equ3-1003-1}.  Then there is $\gamma\in (0,1]$ and $L>0$ so that
\begin{align}\label{equ-2-23-1}
\frac{|E\bigcap [0,x]|}{x} \geq \gamma \quad \mbox{ for all } x\geq L.
\end{align}
Let $\delta:=\gamma/2\in(0,1/2]$. It follows from \eqref{equ-2-23-1} that
\begin{align}\label{equ-2-23-2}
\frac{|E\bigcap [\delta x,x]|}{x} \geq \gamma/2 \quad \mbox{ for all } x\geq L.
\end{align}
We will use \eqref{equ-2-23-2} later. Now we let
$k_1:=\max\left\{\widetilde{k}_0,\, \hat k_0\right\}$,
 where $\widetilde{k}_0$ and $\hat k_0$ are given by
Lemma \ref{lem3.1} and Lemma \ref{lem-930-C} respectively.
Set
\begin{align}\label{equ-2-23-3}
I^\delta_k:=[\delta x_k, x_k],\;\;k\geq k_1,
\end{align}
where $x_k$ is given by \eqref{equ-1025-3}.
Similar to \eqref{equ-1025-2.5}, we can find  $k_2>k_1$ so that when $k\geq k_2$,
\begin{align}\label{equ-2-23-4}
I^\delta_k\subset \left(\Omega_k\cap[x_0+1,\infty)\cap[L,\infty)\right),
\end{align}
where $\Omega_k$ and $x_0$ are given by \eqref{equ3.2} and \eqref{3.29,11-5}, respectively.
The rest of the proof is organized by several steps.

\noindent{\it Step 1. Given  $k\geq  k_2$ and  $\varepsilon\in (0,1)$, we define
\begin{align}\label{equ3.18.1}
F_{k,\varepsilon}: = \left\{ x\in I^\delta_k:\,\,  \cos^2{\left(S(x)-\theta_0\right)}  \leq \varepsilon  \right\},
\end{align}
where $S(\cdot)$  and $\theta_0$ are given by \eqref{equ3.4} and \eqref{equ-1025-7} respectively.
We claim that there is  a constant $C_1>0$, independent of $\varepsilon$ and $k$, so that
\begin{align}\label{equ3-1003-3}
|F_{k,\varepsilon}|\leq  C_1\sqrt{\varepsilon}\lambda_k^{\frac{1}{2c}}\;\;\mbox{for all}\;\;k\geq k_2\;\;\mbox{and}\;\; \varepsilon\in (0,1).
\end{align}}
For this purpose, we  arbitrarily fix $k\geq k_2$ and $\varepsilon\in (0,1)$. Define the following set:
\begin{align}\label{equ3-1003-4}
J':= \left\{ j\in \mathbb{N}: S(\delta x_k)-\pi\leq j\pi+\frac{\pi}{2}+\theta_0\leq S(x_k)+\pi\right\}.
\end{align}
Several observations are given in order. First, by a very similar way to that used in the proof of \eqref{equ-1020-3},
we can obtain
\begin{align}\label{equ-1020-7}
\sharp J' \leq \frac{S(x_k)-S(\delta x_k)}{\pi}+2\leq \frac{1}{2}\alpha\lambda_k^{\frac{1}{2}+\frac{1}{2c}}+2.
\end{align}
 Second, by \eqref{equ3.18.1} and \eqref{equ3-1003-4},  we have
\begin{align}\label{equ-1021-1}
F_{k,\varepsilon}\subset\bigcup_{j\in J'}\left\{x\in I^\delta_k:\,\,\,j\pi+\frac{\pi}{2}-\arcsin{\sqrt{\varepsilon}}\leq S(x)-\theta_0\leq j\pi+\frac{\pi}{2}+\arcsin{\sqrt{\varepsilon}}\right\}.
\end{align}
Third, there is  $C_2>0$ (independent of $k$ and $\varepsilon$) so that when $j\in J'$,
\begin{align}\label{equ3.18.4}
&\left|\left\{x\in I^\delta_k:\,\,\,j\pi+\frac{\pi}{2}-\arcsin{\sqrt{\varepsilon}}\leq S(x)-\theta_0\leq j\pi+\frac{\pi}{2}+\arcsin{\sqrt{\varepsilon}}\right\}\right|\nonumber\\
&=S^{-1}(j\pi+\frac{\pi}{2}+\theta_0+\arcsin{\sqrt{\varepsilon}})-S^{-1}(j\pi+\frac{\pi}{2}+\theta_0-\arcsin{\sqrt{\varepsilon}})\nonumber\\
&\leq 2(\arcsin{\sqrt{\varepsilon}})\cdot\sup_{x\in I^\delta_k}\frac{1}{\sqrt{\lambda_k-V(x)}}\nonumber\\
&\leq C_2 \sqrt{\varepsilon}  \lambda_k^{-\frac{1}{2}}.
\end{align}
In \eqref{equ3.18.4}, for the first equality, Line 2, we used the fact that $S(\cdot)$ is continuous and  strictly increasing on $I_k$ (which follows from \eqref{equ-2-23-4} and \eqref{3.29,11-5}); for the first inequality, Line 3, we used the rule of the derivative of inverse function and the fact that
$S'(x)=\sqrt{\lambda_k-V(x)}$ (which follows from \eqref{equ3.4}); for the last inequality, Line 4, we used  the fact $V(x)\leq \lambda_k/2$ for $x\in I^\delta_k$ (which follows from \eqref{equ-2-23-4} and \eqref{equ3.2}) and  $\arcsin{\sqrt{\varepsilon}}\sim \sqrt{\varepsilon}$.

According to  \eqref{equ-1020-7}-\eqref{equ3.18.4}, there is $C_1>0$ (independent of $k$ and $\varepsilon$) so that
$$
|F_{k,\varepsilon}|\leq  \sharp J'  C_2 \sqrt{\varepsilon}  \lambda_k^{-\frac{1}{2}}\leq C_1\sqrt{\varepsilon} \lambda_k^{\frac{1}{2c}}\;\;\mbox{for all}\;\;k\geq k_2\;\;\mbox{and}\;\; \varepsilon\in (0,1),
$$
which leads to \eqref{equ3-1003-3}.

\noindent {\it Step 2. We prove \eqref{equ3-1003-2}.}

Several observations are given in order. First, by \eqref{equ-2-23-2} (where $x=x_k$),
\eqref{equ-2-23-3} and \eqref{equ-2-23-4}, we find
\begin{align}\label{equ-1026-0}
|E\bigcap I^\delta_k|\ge \frac{\gamma}{2} \alpha\lambda_k^{\frac{1}{2c}},\;\;\mbox{when}\;\; k\ge k_2.
\end{align}
Write
\begin{eqnarray*}\label{equ-1026-1}
\varepsilon_0:= \Big(\frac{1}{1+ {4C_1}/{(\gamma \alpha)}}\Big)^2 (\in (0,1)).
\end{eqnarray*}
 Then it follows from  \eqref{equ3-1003-3} that
\begin{align}\label{equ-1026-2}
|F_{k,\varepsilon_0}|\leq C_1\sqrt{\varepsilon_0}\lambda_k^{\frac{1}{2c}}\leq \frac{\gamma\alpha}4\lambda_k^{\frac{1}{2c}}
 \;\;\mbox{for all}\; k\geq  k_2.
\end{align}
Combining \eqref{equ-1026-0} and \eqref{equ-1026-2}, we get
\begin{align}\label{equ3.18.6}
\left|(E\bigcap I^\delta_k)\setminus F_{k,\varepsilon_0}\right|\geq \frac{\gamma\alpha}{4}\lambda_k^{\frac{1}{2c}},\;\;\mbox{when}\;\; k\geq k_2.
\end{align}
Second,  by \eqref{equ3.18.1} (where $\varepsilon=\varepsilon_0$), we have that
when $k\geq  k_2$,
\begin{align}\label{equ3.18.7}
\cos^2(S(x)-\theta_0) \geq \varepsilon_0\;\;\mbox{for all}\;\;x\in (E\bigcap I^\delta_k)\setminus F_{k,\varepsilon_0}.
\end{align}
Third, noting that when $k\geq k_2$, we have $I^\delta_k\subset \Omega_k$ (see \eqref{equ-2-23-4}),  then  using  \eqref{3.21,11-6}, \eqref{equ3.5} and \eqref{equ3-1003-5}, we obtain that when $k\geq k_2$,
\begin{multline}\label{equ3.22}
\int_{E}{|\varphi_k(x)|^2\,\mathrm dx}\ge\int_{E\cap  I^\delta_k}{|\varphi_k(x)|^2\,\mathrm dx}\\
\ge \frac{1}{2}\int_{E\cap I^\delta_k}{(\lambda_k-V)^{-\frac{1}{2}}|C_{\lambda_k}|^2\cos^2({S(x)-\theta_0})\,\mathrm dx}-3C\int_{E\cap I^\delta_k}{(\lambda_k-V)^{-\frac12}\cdot|C_{\lambda_k}|^2\cdot\lambda_k^{-1}\,\mathrm dx},
\end{multline}
where $C$ is given by \eqref{equ3.5}.

Next, we will estimate two terms on the right hand side of \eqref{equ3.22}, with the aid of
Lemma \ref{lem3.1}, Lemma \ref{lem-930-C} and the first two facts above-mentioned.
First, since  $V\geq 0$ over $I^\delta_k$ (see\eqref{equ3.2.1} and \eqref{equ-2-23-4}), it follows
from
\eqref{equ3.18.7},
 \eqref{equ3.18.6} and  the lower bound \eqref{equ3.6} that when $k\geq k_2$,
\begin{align}\label{equ3.18.8}
&\int_{E\cap I_k^\delta}{(\lambda_k-V(x))^{-\frac{1}{2}}|C_{\lambda_k}|^2\cos^2({S(x)-\theta_0})\,\mathrm dx}\nonumber\\
&\ge \varepsilon_0\int_{(E\bigcap I_k^\delta)\setminus F_{\varepsilon_0}}{(\lambda_k-V(x))^{-\frac12}|C_{\lambda_k}|^2\,\mathrm dx}\nonumber\\
  &\ge \varepsilon_0\lambda_k^{-\frac12}|C_{\lambda_k}|^2\left|(E\bigcap I_k^\delta)\setminus F_{k,\varepsilon_0}\right|\nonumber\\
&\ge C_3\varepsilon_0\gamma\alpha
\end{align}
for some $C_3>0$ (independent of $k$). Second,
because $V\leq \lambda_k/2$ over $I^\delta_k$ (see \eqref{equ-2-23-4} and \eqref{equ3.2}), it follows from \eqref{equ3.3} and the upper bound \eqref{equ3.15}  that when $k\geq k_2$,
\begin{eqnarray}\label{equ3.22.1}
\int_{E\cap I_k^\delta}{(\lambda_k-V(x))^{-\frac12}|C_{\lambda_k}|^2\lambda_k^{-1}\,\mathrm dx} &\leq&|I_k^\delta|({\lambda_k}/{2})^{-\frac12}|C_{\lambda_k}|^2\lambda_k^{-1} \nonumber\\
 &\leq& C_4\lambda_k^{\frac{1}{2c}}\lambda_k^{-\frac{1}{2}}\lambda_k^{\frac{1}{2}-\frac{1}{2c}} \lambda_k^{-1} = C_4\lambda_k^{-1}
\end{eqnarray}
for some $C_4>0$ (independent of $k$).

Finally, inserting \eqref{equ3.18.8} and \eqref{equ3.22.1} into \eqref{equ3.22}, we find
that when $k\geq  k_2$,
\begin{align}\label{equ-1026-4}
 \int_{E}{|\varphi_k(x)|^2\,\mathrm dx}\geq \frac{1}{2}C_3\varepsilon_0\gamma\alpha - 3C_4\lambda_k^{-1}.
\end{align}
Since $\lambda_k^{-1}\rightarrow 0$ as $k\rightarrow \infty$, we can find $k_0\geq   k_2$ so that
when $k\geq k_0$,
\begin{eqnarray*}
\frac{1}{2}C_3\varepsilon_0\gamma\alpha - 3C_4\lambda_k^{-1} \geq \frac{1}{4}C_3\varepsilon_0\gamma\alpha,
\end{eqnarray*}
which, together with \eqref{equ-1026-4}, leads to \eqref{equ3-1003-2}.

Hence, we end the proof of Lemma \ref{lem-low-h}.
\end{proof}

\subsection{Proof of Theorem \ref{thm-ob-for-genral}}

Arbitrarily fix a subset $E\subset \mathbb{R}$ satisfying \eqref{equ-930-4}. Then we have $|E|>0$.

We first claim that there exists $\varepsilon_0>0$ so that
\begin{align}\label{equ-930-2}
\lambda_{k+1}-\lambda_k\geq \varepsilon_0\;\;\mbox{ for all }\;\; k\in \mathbb{N}^+.
\end{align}
In fact, because of  \textbf{Condition (H)},  we can apply  \cite[Lemma 3.3]{Ya}
to find $k_0\in \mathbb{N}^+$ and $C>0$, which are independent of $k$, so that
  \begin{eqnarray*}
  \lambda_{k+1}-\lambda_k \ge C\lambda_k^{\frac{1}{2}-\frac{1}{2c}}\;\;\mbox{for all}\;\;k\geq k_0,
  \end{eqnarray*}
  which, along with the conclusion $(ii)$ in Proposition \ref{prop-H}, leads to
  \eqref{equ-930-2}.

Next, we claim that there exists $C>0$, independent of $k$, so that
\begin{align}\label{equ-930-3}
\int_E|\varphi_k(x)|^2\d x \geq C\;\;\mbox{ for all }\;\; k\in \mathbb{N}^+.
\end{align}
Indeed, by \eqref{equ-930-4}, we can apply  Lemma \ref{lem-low-h}
to find $k_0\in \mathbb{N}^+$ and $C>0$, independent of $k$, so that \eqref{equ-930-3} holds for all $k\geq k_0$.
This, together with Lemma \ref{lem-low-f}, leads to \eqref{equ-930-3}.

Finally, by \eqref{equ-930-2} and \eqref{equ-930-3}, we can use  Proposition \ref{pro-ob-general}
to see that $E$ is \emph{an observable set at some time} for the equation \eqref{equ3.8}. This ends the proof of Theorem \ref{thm-ob-for-genral}.\qed

\section{Proof of Theorem \ref{thm-ob-for-HO}}\label{sec4}
In this section, we mainly prove Theorem \ref{thm-ob-for-HO}, besides, we give the difference between
\emph{thick sets} and \emph{weakly thick sets}.
The proof of Theorem \ref{thm-ob-for-HO} is based on  the following two theorems:

\begin{theorem}\label{thm-4-suff}
Let $E\subset \mathbb{R}$ be   weakly thick. Then the following conclusions are true:

\noindent $(i)$   The set $E$ is  an observable set at some time  for the equation  \eqref{equ-1.0} with   $m=1$.

\noindent $(ii)$   The set $E$ is  an observable set at any time  for the equation  \eqref{equ-1.0} with   $m\ge 2$.

\end{theorem}

\begin{theorem}\label{thm-4-ne}
If  $E$  is an observable set at some time for the equation \eqref{equ-1.0} with $m\in \mathbb{N}^+$,
then it is  weakly thick.
\end{theorem}

\noindent
{\bf Proof of Theorem \ref{thm-ob-for-HO}.}
The equivalence of $(i)$ and $(ii)$ in  Theorem \ref{thm-ob-for-HO} follows from $(i)$ of Theorem \ref{thm-4-suff} and Theorem \ref{thm-4-ne}.  The equivalence of $(i)$ and $(iii)$ in  Theorem \ref{thm-ob-for-HO} follows from $(ii)$ of Theorem \ref{thm-4-suff} and Theorem \ref{thm-4-ne}.
This ends the proof of Theorem \ref{thm-ob-for-HO}. \qed

The rest of this section is organized as follows: Subsection 4.1 provides  some preliminaries;
Subsections 4.2 and 4.3 prove  Theorem \ref{thm-4-suff} and Theorem \ref{thm-4-ne}, respectively;
 Subsection 4.4 presents the difference between
\emph{thick sets} and \emph{weakly thick sets}.

\subsection{Preliminaries}

Arbitrarily fix $m\in \mathbb{N}^+$. Let   $H$ be  given by \eqref{equ-1.0-1130}.
  {\it Recall that $\{\lambda_k\}_{k=1}^\infty$ and $\{\varphi_k\}_{k=1}^\infty$ are the eigenvalues and the corresponding $L^2$-normalized eigenfunctions of the above $H$. We also recall the property \eqref{3.4-10-15}
  for $\{\lambda_k\}_{k=1}^\infty$.}

To show that the weakly thick condition \eqref{def-1-3} is sufficient for \emph{observable sets at some time}
for the equation \eqref{equ-1.0}, we first use  Theorem \ref{thm-ob-for-genral} (where $V(x)=x^{2m}$)
to get the sufficient condition \eqref{equ-930-4} on  \emph{observable sets at some time}
for the equation \eqref{equ-1.0}. (Notice that the  potential $V(x)=x^{2m}$ satisfies \textbf{Condition (H)}.)  Next, we will find connections between \eqref{equ-930-4} and the weakly thick condition \eqref{def-1-3}, through using the property: each eigenfunction is either even or odd.

To prove that the weakly thick condition \eqref{def-1-3} is  necessary for \emph{observable sets at some time}
for the equation \eqref{equ-1.0}, we shall use the following  explicit asymptotic expression of
eigenvalues:  (See e.g. \cite{EGS,S}.)
 \begin{align}\label{equ4.34}
\lambda_k=\left(\frac{\pi}{B(3/2, {1}/{(2m)})}\cdot k\right)^{\frac{2m}{m+1}}(1+r_k), \;\;
k\in \mathbb{N}^+;
\;\;\mbox{and}\;\;\lim_{k\rightarrow +\infty}r_k=0,
\end{align}
where $B(\cdot, \cdot)$ is the Beta function.
From \eqref{equ4.34}, we can directly verify the next Lemma \ref{lemma41.,10-23}:
\begin{lemma}\label{lemma41.,10-23}
The eigenvalues $\{\lambda_k\}_{k=1}^\infty$  satisfy \eqref{equ-929-11}.
\end{lemma}
(It deserves mentioning that  since  \textbf{Condition (H)} holds for $V(x)=x^{2m}$, Lemma \ref{lemma41.,10-23} has been proved in the proof of Theorem \ref{thm-ob-for-genral} (see \eqref{equ-930-2}).

The next two lemmas will be used in the proof
that the weakly thick condition \eqref{def-1-3} is necessary for \emph{observable sets at some time}
for the equation \eqref{equ-1.0} in Subsection \ref{sec4.3-11-2}.
\begin{lemma}\label{lemma4.1-4-17}
Let $E\subset \mathbb{R}$ be a measurable subset.
Then for each $a>0$,
 \begin{align}\label{equ-101-1}
\varliminf_{\mathbb{N}\ni k\rightarrow  \infty} \frac{|E\bigcap[-a\lambda_k^{\frac{1}{2m}},a\lambda_k^{\frac{1}{2m}}]|}{a\lambda_k^{\frac{1}{2m}}} = \varliminf_{\mathbb{N}\ni k\rightarrow  \infty} \frac{|E\bigcap[-ab k^{\frac{1}{m+1}},abk^{\frac{1}{m+1}}]|}{abk^{\frac{1}{m+1}}},
\end{align}
where
\begin{align}\label{4.4-10-17}
 b=\left(\frac{\pi}{B(3/2, {1}/{2m})}\right)^{\frac{1}{m+1}}.
\end{align}
 \end{lemma}
\begin{proof}
First of all, by \eqref{equ4.34}, we have
\begin{align}\label{equ-2-22-1}
 bk^{\frac{1}{m+1}}(1-|r_k|)^{\frac{1}{2m}}\leq \lambda_k^{\frac{1}{2m}} \leq  bk^{\frac{1}{m+1}}(1+|r_k|)^{\frac{1}{2m}}\;\;\mbox{for all}\;\; k\in \mathbb{N}^+,
\end{align}
where $b$ is given by \eqref{4.4-10-17}.
 Arbitrarily fix $a>0$. Then by  \eqref{equ-2-22-1}, we find that for all $k\in \mathbb{N}^+$,
\begin{align}\label{equ-2-22-2}
&\left(E \cap [-a\lambda_k^{\frac{1}{2m}},a\lambda_k^{\frac{1}{2m}} ]\right) \subset \left(E\cap[-abk^{\frac{1}{m+1}}(1+|r_k|)^{\frac{1}{2m}},abk^{\frac{1}{m+1}}(1+|r_k|)^{\frac{1}{2m}} ]\right).
\end{align}
From \eqref{equ-2-22-2}, we see
\begin{multline}\label{equ-2-22-3}
   \varliminf_{\mathbb{N}\ni k\rightarrow  \infty}\limits \frac{|E\bigcap [-a\lambda_k^{\frac{1}{2m}}, \, a\lambda_k^{\frac{1}{2m}}]|}{a\lambda_k^{\frac{1}{2m}}}
\leq \varliminf_{\mathbb{N}\ni k\rightarrow  \infty}\limits \frac{|E\bigcap [-abk^{\frac{1}{m+1}}(1+|r_k|)^{\frac{1}{2m}}, \, abk^{\frac{1}{m+1}}(1+|r_k|)^{\frac{1}{2m}}]|}{abk^{\frac{1}{m+1}}(1-|r_k|)^{\frac{1}{2m}}}  \\
  \leq \varliminf_{\mathbb{N}\ni k\rightarrow  \infty}\limits \frac{|E\bigcap [-abk^{\frac{1}{m+1}}, abk^{\frac{1}{m+1}}]|}{abk^{\frac{1}{m+1}}(1-|r_k|)^{\frac{1}{2m}}} + \varliminf_{\mathbb{N}\ni k\rightarrow  \infty}\limits \frac{  2\big((1+|r_k|)^{\frac{1}{2m}}-1\big) }{(1-|r_k|)^{\frac{1}{2m}}}.
\end{multline}
Since $\lim_{\mathbb{N}\ni k\rightarrow  \infty}r_k=0$ (see \eqref{equ4.34}), the second term on the right hand side of \eqref{equ-2-22-3}  vanishes. So we obtain from \eqref{equ-2-22-3} that
\begin{align}\label{equ-2-22-4}
\varliminf_{\mathbb{N}\ni k\rightarrow  \infty}\limits \frac{|E\bigcap [-a\lambda_k^{\frac{1}{2m}}, \, a\lambda_k^{\frac{1}{2m}}]|}{a\lambda_k^{\frac{1}{2m}}} \leq \varliminf_{\mathbb{N}\ni k\rightarrow  \infty}\limits \frac{|E\bigcap [-abk^{\frac{1}{m+1}}, abk^{\frac{1}{m+1}}]|}{abk^{\frac{1}{m+1}}}.
\end{align}
On the other hand, it follows from \eqref{equ-2-22-1} that when $k$ is large enough so that $|r_k|<1$,
\begin{align}\label{equ-2-22-5}
&\left(E \cap [-a\lambda_k^{\frac{1}{2m}},a\lambda_k^{\frac{1}{2m}} ]\right) \supset \left(E\cap[-abk^{\frac{1}{m+1}}(1-|r_k|)^{\frac{1}{2m}},abk^{\frac{1}{m+1}}(1-|r_k|)^{\frac{1}{2m}} ]\right).
\end{align}
Similar to \eqref{equ-2-22-4}, one can deduce from \eqref{equ-2-22-5} that
\begin{align}\label{equ-2-22-6}
\varliminf_{\mathbb{N}\ni k\rightarrow  \infty}\limits \frac{|E\bigcap [-a\lambda_k^{\frac{1}{2m}}, \, a\lambda_k^{\frac{1}{2m}}]|}{a\lambda_k^{\frac{1}{2m}}} \geq \varliminf_{\mathbb{N}\ni k\rightarrow  \infty}\limits \frac{|E\bigcap [-abk^{\frac{1}{m+1}}, abk^{\frac{1}{m+1}}]|}{abk^{\frac{1}{m+1}}}.
\end{align}
Finally, \eqref{equ-101-1} follows from  \eqref{equ-2-22-4} and \eqref{equ-2-22-6} at once.
This ends the proof of Lemma \ref{lemma4.1-4-17}.
\end{proof}

The next Proposition \ref{prop-915} gives different equivalent versions  of the characterization of \emph{ weakly thick sets}, in  particular, it  shows the connection between \eqref{equ-101-1} and \eqref{def-1-3}.

\begin{proposition}\label{prop-915}
Let $E\subset \mathbb{R}$ be a measurable set. Then the following statements are equivalent:

\noindent $(i)$ The set $E$ is weakly thick, i.e.,  $\varliminf_{x \rightarrow +\infty}\limits \frac{|E\bigcap [-x,x]|}{x} >0$.

\noindent $(ii)$ The set $E$ satisfies that $\varliminf_{\mathbb{N}\ni k \rightarrow \infty}\limits \frac{|E\bigcap [-k,k]|}{k} >0$.

\noindent $(iii)$ For all $a>0$ and $l>0$,  $\varliminf_{\mathbb{N}\ni k \rightarrow  \infty}\limits \frac{|E\bigcap [-ak^l,ak^l]|}{ak^l} >0$.

\noindent $(iv)$ There is $a>0$ and $l>0$ so that  $\varliminf_{\mathbb{N}\ni k \rightarrow  \infty}\limits \frac{|E\bigcap [-ak^l,ak^l]|}{ak^l} >0$.
\end{proposition}
\begin{proof}
It is clear that $(i)\Longrightarrow (ii)$ and $(iii)\Longrightarrow (iv)$. The rest of the proof is organized by two steps.

\noindent {\it Step 1. We show that  $(ii)\Longrightarrow (iii)$.}

By $(ii)$, there is  $\gamma>0$ and $k_0\in\mathbb{N}^+$ so that
\begin{align}\label{equ-916-1}
\frac{|E\cap [-k,k]|}{k}\geq \gamma\;\;\mbox{ for all }\;\;  k\geq k_0.
\end{align}
Let $a>0$ and $l>0$. Arbitrarily fix $k\in \mathbb{N}^+$ so that
\begin{align}\label{equ-916-2}
k\geq c_1 := \left( \frac{k_0+1}{a} \right)^{1/l}.
\end{align}
Then  there exists a unique  $n\in\mathbb{N}$ so that
\begin{align}\label{equ-916-3}
n\leq ak^l<n+1.
\end{align}
From  \eqref{equ-916-2}, we have $n>k_0$. Meanwhile, from \eqref{equ-916-3}, we obtain
$$
E\cap [-ak^l,ak^l]\supset  E\cap [-n, n].
$$
Thus, by \eqref{equ-916-1}, we find that for any $k\in \mathbb{N}^+\cap [c_1,+\infty)$,
\begin{align*}
|E \cap [-ak^l,ak^l]|\geq |E\cap [-n, n]|\geq  \gamma n.
\end{align*}
This, along with  \eqref{equ-916-3}, yields that
  $$
\varliminf_{\mathbb{N}\ni k \rightarrow \infty} \frac{|E \cap [-ak^l,ak^l]|}{ak^l} \geq \varliminf_{\mathbb{N}\ni k \rightarrow \infty} \frac{\gamma(ak^l-1)}{ak^l} =\gamma>0,
  $$
 which leads to  $(iii)$.

  \noindent {\it Step 2. We show that $(iv)\Longrightarrow (i)$.}

  By $(iv)$, there is $k_0\in \mathbb{N}^+$ and $\gamma>0$ so that
  \begin{align}\label{equ-916-6}
\frac{|E\bigcap [-ak^l,ak^l]|}{ak^{l}}\geq \gamma\;\; \mbox{ for all } \; k\geq k_0.
\end{align}
Arbitrarily fix $x\in \mathbb{R}$ so that
  \begin{align}\label{equ-916-7}
x\geq c_2:=  a(k_0+1)^l.
\end{align}
  Then, there exists a unique  $n\in \mathbb{N}$ so that
\begin{align}\label{equ-916-8}
 an^l\leq x< a(n+1)^l.
\end{align}
The fact \eqref{equ-916-7}, together with \eqref{equ-916-8}, implies that $n>k_0$.
Meanwhile, from \eqref{equ-916-8}, we also have
$$
E\cap[-x, x]\supset E\cap [-an^l,an^l],
$$
which, along with \eqref{equ-916-6}, leads to
\begin{align}\label{equ-916-10}
|E\cap[-x, x]|\geq | E\cap [-an^l,an^l]| \geq   \gamma an^l.
\end{align}
Now, from   \eqref{equ-916-10} and \eqref{equ-916-8}, we find that for all $x\in [c_2,+\infty)$,
\begin{eqnarray}\label{equ-916-11}
|E\cap[-x, x]|   \geq \gamma(\frac{n}{n+1})^l x\geq \gamma 2^{-l} x.
\end{eqnarray}
It follows from \eqref{equ-916-11} that
   $$
\varliminf_{x \rightarrow +\infty} \frac{|E\cap [-x,x]|}{x} \geq \gamma 2^{-l}>0,
  $$
which leads to $(i)$.

Hence, we  complete the proof of Proposition \ref{prop-915}.
\end{proof}

\subsection{Proof  of Theorem \ref{thm-4-suff}}\label{subsect4.2}

Arbitrarily fix \emph{a weakly thick set} $E\subset \mathbb{R}$. Then we have
\begin{align}\label{equ-2-23-5}
\varliminf_{x \rightarrow +\infty}\limits \frac{|E\bigcap [-x,x]|}{x} >0.
\end{align}
Arbitrarily fix $m\in \mathbb{N}^+$.
First of all,
we  claim that the eigenfunctions $\{\varphi_k\}_{k\in \mathbb{N}^+}$ (to the operator $H=-\partial_x^2+x^{2m}$) satisfy
\begin{align}\label{equ-2-24-1}
\int_{E}|\varphi_k(x)|^2\d x \geq C\;\;\mbox{for all}\; k\in \mathbb{N}^+,
\end{align}
where $C>0$ is independent of $k$.
To this end, we introduce the following sets:
\begin{align}\label{equ-2-23-6}
E_+ := E\cap [0,\infty);\quad  E_-:= E\cap(-\infty,0);\quad E^*_-:= \{-x,\,x\in E_-\}.
\end{align}
It is clear that  $E=E_+\bigcup E_-$ and $E_+\bigcap E_-=\emptyset$ and that
$$
E\cap[-x,x] = \Big(E_+ \cap [0,x]\Big) \bigcup \Big(E_- \cap [-x,0)\Big)\;\;\mbox{for each}\;\;x>0,
$$
which, together with \eqref{equ-2-23-6}, implies that
\begin{align}\label{equ-2-23-6.5}
|E\cap[-x,x]| &= |E_+ \cap [0,x]| +|E_- \cap [-x,0)|= |E_+ \cap [0,x]| +|E^*_- \cap [0,x]|\nonumber\\
&\leq 2\left|\Big(E_+\cup E^*_-\Big) \cap [0,x]\right|=2\left|\tilde{E} \cap [0,x]\right|,
\end{align}
where
\begin{eqnarray}\label{equ-2-23-7}
\tilde{E}:= E_+\cup E^*_-\subset [0,\,\infty).
\end{eqnarray}
Now it follows from \eqref{equ-2-23-5} and \eqref{equ-2-23-6.5} that
\begin{align}\label{equ-2-23-8}
\varliminf_{x \rightarrow +\infty}\limits \frac{|\tilde{E}\bigcap [0,x]|}{x} >0.
\end{align}
Because of \eqref{equ-2-23-8}, we can apply Theorem \ref{thm-ob-for-genral} to conclude that $\tilde{E}$ is \emph{an observable set at some time} for \eqref{equ-1.0}. Further, according to Proposition \ref{pro-ob-general},  there is   $C>0$ (independent of $k$) so that
\begin{align}\label{equ-10-5-1706}
\int_{\tilde{E}}|\varphi_k(x)|^2\d x \geq C\;\;\mbox{for all}\; k\in \mathbb{N}^+.
\end{align}

To proceed, we  need the following  {\bf Key Observation}: {\it Each  eigenfunction of $H$ is either  even or odd. }
Indeed, we have
\begin{align}\label{equ-10-5}
\Big(-\frac{d^2}{dx^2}+x^{2m}\Big)\varphi_k(x) = \lambda_k \varphi_k(x), \qquad  x\in \mathbb{R}.
\end{align}
Let $\widetilde{\varphi_k}(x):= \varphi_k(-x)$, $x\in \mathbb{R}$. One can easily check that $\widetilde{\varphi_k}$ also satisfies \eqref{equ-10-5} and that $\|\widetilde{\varphi_k}\|_{L^2(\mathbb{R})}=\|\varphi_k\|_{L^2(\mathbb{R})}$. These, along with
the conclusion $(ii)$  Proposition \ref{prop-H}, give immediately that either
$\varphi_k=\widetilde{\varphi_k}$ or $\varphi=-\widetilde{\varphi_k}$, which leads to {\bf Key Observation}.

By {\bf Key Observation}  and \eqref{equ-2-23-6}, we infer that for all $k\in \mathbb{N}^+$,
\begin{align}\label{equ-10-5-1702}
\int_{E_-}|\varphi_k(x)|^2\d x=\int_{E^*_-}|\varphi_k(x)|^2\d x.
\end{align}
It follows from \eqref{equ-2-23-6}, \eqref{equ-10-5-1702} and \eqref{equ-2-23-7} that
\begin{align}\label{equ-10-5-1704}
\int_{E}|\varphi_k(x)|^2\d x=\int_{E_+}|\varphi_k(x)|^2\d x+\int_{E_-}|\varphi_k(x)|^2\d x\ge \int_{\tilde{E}}|\varphi_k(x)|^2\d x.
\end{align}
Combining \eqref{equ-10-5-1706} and \eqref{equ-10-5-1704}, we find that
$$
\int_{E}|\varphi_k(x)|^2\d x \geq C\;\;\mbox{for all}\; k\in \mathbb{N}^+,
$$
which leads to  \eqref{equ-2-24-1}.

Now we prove $(i)$ of Theorem \ref{thm-4-suff}.
Indeed, by  \eqref{equ-2-24-1} and Lemma \ref{lemma41.,10-23}, we can use  Proposition \ref{pro-ob-general}   to conclude that $E$ is \emph{an observable set at some time} for \eqref{equ-1.0} with  $m \in \mathbb{N}^+$, which leads to $(i)$ clearly.

Next we prove $(ii)$ of Theorem \ref{thm-4-suff}.
Arbitrarily fix $m\geq 2$. By
  \eqref{equ4.34}, after some  direct calculation, we can find $C>0$ (independent of $k$) so that
$$
\lambda_{k+1}-\lambda_{k}  \geq ~ Ck^{\frac{m-1}{m+1}}\rightarrow \infty,\,\,\,\,\text{as}\,\,\,k\rightarrow\infty.
$$
Then because of \eqref{equ-2-24-1} and Lemma \ref{lemma41.,10-23}, we apply the last statement in Proposition \ref{pro-ob-general}  to conclude that $E$ is \emph{an observable set at any time} for \eqref{equ-1.0}. This completes the proof of Theorem \ref{thm-4-suff}. \qed

\subsection{Proof of  Theorem \ref{thm-4-ne}}\label{sec4.3-11-2}
 Let $E$ be \emph{an observable set at some time $T_0>0$} for \eqref{equ-1.0}, with an arbitrarily fixed $m\in \mathbb{N}^+$ . Then there is  $C_0=C_0(T_0,E)>0$ so that
\begin{align}\label{equ-10-5-1709}
\|u_0\|_{L^2}^2\leq C_0\int_0^{T_0}\int_{E}{|e^{-\i tH}u_0|^2\d x\d t}\;\;\mbox{for all}\;\;u_0\in L^2(\mathbb{R}).
\end{align}
By taking $u_0=\varphi_k$ in \eqref{equ-10-5-1709} and noting that $\varphi_k$ is
the $L^2$ normalized eigenfunction of $H$, we find
\begin{align}\label{equ-10-5-1710}
\int_{E}{|\varphi_k|^2 \d x}\ge C_1\;\;\mbox{ for all }\;\; k\in \mathbb{N}^+,
\end{align}
where $C_1=1/(T_0C_0)$. (Notice that \eqref{equ-10-5-1710} can also be obtained by Proposition \ref{pro-ob-general}.)

In order to show that $E$ is \emph{weakly thick}  from the uniform inequality \eqref{equ-10-5-1710}, the asymptotic expression \eqref{3.21,11-6}
  for general potentials with \textbf{Condition (H)} doesn't seem  to be enough.
  We need  a finer asymptotic expression of $\varphi_k$ for the case
  \eqref{equ-1.0-1130}. This will be given
   by the next Lemma \ref{lem-append}. To state it, we write
  \begin{eqnarray}\label{4.17,10-23}
  \mu_k:=  \lambda_k^{\frac{1}{2m}} \;\;\mbox{for each}\;\; k>>1\; (\mbox{so that}\;\;\lambda_k>0).
  \end{eqnarray}
  For each $x\in \mathbb{R}$ and each $k>>1$, we define
  \begin{align}\label{eq-app-5}
S_k^-(x):= \int_0^x\sqrt{|\mu_k^{2m}-t^{2m}|}\,\d t\;\;\mbox{and}\;\;S_k^+(x):= \int_{\mu_k}^{|x|}\sqrt{|t^{2m}-\mu_k^{2m}|}\,\d t.
\end{align}
 We also notice that $\varphi_k$ satisfies
\begin{equation}\label{eq-app-1}
\begin{cases}
-\varphi_k''(x)+x^{2m}\varphi_k(x)=\mu_k^{2m} \varphi_k(x),\;\;x\in \mathbb{R};\\
\|\varphi_k\|_{L^2(\mathbb{R})}=1.
\end{cases}
\end{equation}
 The next Lemma \ref{lem-append} is the key in our  proof.
\begin{lemma}\label{lem-append}
With  notations in \eqref{4.17,10-23} and \eqref{eq-app-5}, when $k\rightarrow +\infty$,
either
\begin{equation}\label{eq-app-2}
\varphi_{k}(x)=\begin{cases}
a_{2k}^-(\mu_k^{2m}-|x|^{2m})^{-\frac14}(\cos{S_k^-(x)}+R_k(x)),\,\,\,|x|<\mu_k- \delta\cdot\mu_k^{-\frac{2m-1}{3}},\\
O(\mu_k^{\frac{m-2}{6}}),\,\,\,\,\,\,\,\,\,\,\,\quad\quad\quad\quad \quad\mu_k-\delta\cdot\mu_k^{-\frac{2m-1}{3}}\leq|x|\leq \mu_k+ \delta\cdot\mu_k^{-\frac{2m-1}{3}},\\
a_{2k}^+(|x|^{2m}-\mu_k^{2m})^{-\frac14}e^{-S_k^+(x)}(1+R_k(x)),\,\,\,|x|>\mu_k+ \delta\cdot\mu_k^{-\frac{2m-1}{3}},
\end{cases}
\end{equation}
or
\begin{equation}\label{eq-app-3}
\varphi_{k}(x)=\begin{cases}
a_{2k+1}^-(\mu_k^{2m}-|x|^{2m})^{-\frac14}(\sin{S_k^-(x)}+R_k(x)),\,\,\,|x|<\mu_k- \delta\cdot\mu_k^{-\frac{2m-1}{3}},\\
O(\mu_k^{\frac{m-2}{6}}),\,\,\,\,\,\,\,\,\,\,\,\quad\quad\quad\quad \quad\mu_k-\delta\cdot\mu_k^{-\frac{2m-1}{3}}\leq|x|\leq \mu_k+ \delta\cdot\mu_k^{-\frac{2m-1}{3}},\\
a_{2k+1}^+(|x|^{2m}-\mu_k^{2m})^{-\frac14}e^{-S_{k}^+(x)}(1+R_k(x)),\,\,\,|x|>\mu_k+ \delta\cdot\mu_k^{-\frac{2m-1}{3}},
\end{cases}
\end{equation}
where $\delta>0$ is independent of $k$ and $x$ and
\begin{align}\label{eq-app-4}
|a_k^{\pm}|\sim \mu_k^{\frac{m-1}{2}},\,\,\,\,R_k(x)=O\left(\left|x^{2m}-\mu_k^{2m}\right|^{-\frac12}\cdot ||x|-\mu_k|^{-1}\right).
\end{align}
\end{lemma}
Here and in what follows, given sequences of numbers $\{\alpha_k\}$ and $\{\gamma_k\}$, by $\alpha_k\sim \gamma_k$, we mean that
there is $C_1>0$ and $C_2>0$ so that $C_1|\gamma_k|\leq |\alpha_k|\leq C_2|\gamma_k|$ for all $k$, while by $\alpha_k=O(\gamma_k)$,
we mean that
there is $C_3>0$ so that $|\alpha_k|\leq C_3|\gamma_k|$ for all $k$.

\begin{remark}\label{remark4.5,10-23}
One can use the standard WKB method (see e.g. in \cite{BS,Fe,Ti}) to obtain asymptotic expressions of the form $\varphi_k=f(x)e^{iS(x)}$ for certain amplitude $f$ and phase function $S$. The corresponding result for the case $m=1$ was stated in \cite[Lemma 5.1]{KT} without  proof. Since  Lemma \ref{lem-append} will play  an important role in our proof, we
will give its detailed proof  in the Appendix \ref{sec-app} for the sake of completeness of the paper.
\end{remark}

We now back to the proof of Theorem \ref{thm-4-ne}.  In order to apply  Lemma \ref{lem-append}, we make  the following decomposition:
\begin{eqnarray}\label{equ3.62}
\int_{E}|{\varphi_k|^2\,\d x}=I_1+I_2+I_3,
\end{eqnarray}
where
\begin{align*}
I_1&:= \int_{E\bigcap\Big\{x:|x|< \mu_k-\delta\mu_k^{-\frac{2m-1}{3}}\Big\}}|{\varphi_k(x)|^2\,\mathrm dx}{,}\\
I_2&:=\int_{E\bigcap\Big\{x:\mu_k-\delta\mu_k^{-\frac{2m-1}{3}}\leq|x|\leq \mu_k+\delta\mu_k^{-\frac{2m-1}{3}}\Big\}}|{\varphi_k(x)|^2\,\mathrm dx}{,}\\
I_3&:=\int_{E\bigcap\Big\{x:|x|>\mu_k+\delta\mu_k^{-\frac{2m-1}{3}}\Big\}}|{\varphi_k(x)|^2\,\mathrm dx}.
\end{align*}

To deal with the term $I_1$, we observe that by \eqref{eq-app-4}, there is
$C>0$ (independent of $k$ and $x$) so that
\begin{align}\label{eq-reminder-es}
|R_k(x)|\leq  C\left(\mu_k^{2m-1}\mu_k^{-\frac{2m-1}{3}}\right)^{-\frac12} \mu_k^\frac{2m-1}{3}\leq C,
\end{align}
when $x$ satisfies
\begin{eqnarray*}
\mbox{either}\;\;|x|<\mu_k-\delta\mu_k^{-\frac{2m-1}{3}}\;\;\mbox{or}
\;\;|x|>\mu_k+\delta\mu_k^{-\frac{2m-1}{3}}.
\end{eqnarray*}
Now, we write
\begin{align}\label{equ-10-5-1711}
I_1&=  \int_{{E\bigcap\Big\{x:|x|< \rho\mu_k\Big\}}}{|\varphi_k|^2\,\mathrm dx}+ \int_{{E\bigcap\Big\{x:\rho\mu_k<|x|< \mu_k-\delta\mu_k^{-\frac{2m-1}{3}}\Big\}}}{|\varphi_k|^2\,\mathrm dx}\nonumber\\
&:= I_{1,1}+I_{1,2},
\end{align}
where $\rho\in(0, 1)$ is some constant to be chosen later.  Notice that for any given $\rho\in(0, 1)$, we have
$$
\rho\mu_k<\mu_k-\delta\mu_k^{-\frac{2m-1}{3}},\,\,\,\text{as }\,\,k \to \infty.
$$
Thus we can use \eqref{eq-reminder-es}, as well as Lemma \ref{lem-append}, to find   $k_1\in\mathbb{N}^+$ so that when $k\geq k_1$,
\begin{align}\label{equ-10-5-1712}
I_{1,1}&\leq  C_2 \int_{E\bigcap\{x:|x|< \rho\mu_k\}}{\mu_k^{m-1}(\mu_k^{2m}-x^{2m})^{-\frac12}\,\mathrm dx}\nonumber\\
&\leq C_2(1-\rho^{2m})^{-\frac12}\frac{|E\bigcap\{x:|x|< \rho\mu_k\}|}{\mu_k}
\end{align}
and
\begin{align}\label{equ-10-5-1713}
I_{1,2}&\leq  C_3 \int_{{\Big\{x:\rho\mu_k<|x|< \mu_k-\delta\mu_k^{-\frac{2m-1}{3}}\Big\}}}{\mu_k^{m-1}(\mu_k^{2m}-x^{2m})^{-\frac12}\,\mathrm dx}\nonumber\\
&\leq C_3\int_{\rho}^1{(1-x^{2m})^{-\frac12}\,\mathrm dx},
\end{align}
where $C_2, C_3>0$ are two absolute constants. Since $\int_{\rho}^1{(1-x^{2m})^{-\frac12}\,\mathrm dx}\rightarrow 0$ as $\rho\rightarrow 1^{-}$,  we can choose $\rho=\rho_0\in (0, 1)$ so that
\begin{align}\label{equ-10-5-1714}
\int_{\rho_0}^1{(1-x^{2m})^{-\frac12}\,\mathrm dx}<\frac{C_1}{100C_3}.
\end{align}
Then  it follows from \eqref{equ-10-5-1711}-\eqref{equ-10-5-1714} that
\begin{eqnarray}\label{equ3.63}
I_1
\leq C_2(1-\rho_0^{2m})^{-\frac12}\frac{|E\bigcap\{x:|x|< \rho_0\mu_k\}|}{\mu_k}+\frac{C_1 }{100}.
\end{eqnarray}

For the term $I_2$, we can use  Lemma \ref{lem-append} again to find $k_2\in\mathbb{N}^+$ and $C_4>0$ so that
\begin{align}\label{equ3.64}
I_2\leq C_4\mu_k^{\frac{m-2}{3}}\mu_k^{-\frac{2m-1}{3}}=C_4\mu_k^{-\frac{m+1}{3}}\leq \frac{C_1}{100}\;\;\mbox{for all}\;\;k\ge k_2.
\end{align}

We next deal with the term $I_3$. First we claim that for large $k$,
\begin{align}\label{equ-est-S(x)}
S_k^+(x)>\frac{\sqrt{\delta}}{3}\mu_k^{\frac{2m-1}{3}}\left(|x|-\mu_k\right),
\;\;\mbox{when}\;\;|x|>\mu_k+\delta\mu_k^{-\frac{2m-1}{3}}.
\end{align}
(Here, $S_k^+(x)$ is given by \eqref{eq-app-5}.) Indeed, by \eqref{eq-app-5}, one can directly check
that
\begin{eqnarray}\label{4.30,10-23}
S^+_k\left(\mu_k+\delta\mu_k^{-\frac{2m-1}{3}}\right)\geq \frac{2}{3}\delta^{\frac{3}{2}}\;\;\mbox{for all}\;\;k>>1.
\end{eqnarray}
We define the function:
\begin{eqnarray}\label{4.31,10-23}
F(x)= S_k^+(x)-\frac{\sqrt{\delta}}{3}\mu_k^{\frac{2m-1}{3}}\left(|x|-\mu_k\right),\; x\in \mathbb{R}.
\end{eqnarray}
From \eqref{4.31,10-23} and \eqref{4.30,10-23}, we have that
\begin{eqnarray}\label{4.32,10-23}
F\left(\mu_k+\delta\mu_k^{-\frac{2m-1}{3}}\right)>0\;\;\mbox{for all}\;\;k>>1
\end{eqnarray}
and that when $k\in \mathbb{N}^+$,
\begin{align}\label{4.33,10-23}
F'(x)&=\sqrt{|x^{2m}-\mu_k^{2m}|}-\frac{\sqrt{\delta}}{3}\mu_k^{\frac{2m-1}{3}}\nonumber\\
&\ge\left(\left(\mu_k+\delta\mu_k^{-\frac{2m-1}{3}}\right)^{2m}-\mu_k^{2m}\right)^{\frac12}
-\frac{\sqrt{\delta}}{3}\mu_k^{\frac{2m-1}{3}}\nonumber\\
&\ge\left(\sqrt{2m}-\frac{1}{3}\right)\sqrt{\delta}\mu_k^{\frac{2m-1}{3}}>0,
\;\;\mbox{when}\;\;|x| \geq\mu_k+\delta\mu_k^{-\frac{2m-1}{3}}.
\end{align}
From \eqref{4.33,10-23}, \eqref{4.32,10-23} and \eqref{4.31,10-23}, we obtain \eqref{equ-est-S(x)}.
Now by Lemma \ref{lem-append}, \eqref{eq-reminder-es} and \eqref{equ-est-S(x)}, there is some  constant $C_5>0$
(independent of $k$) so that
\begin{eqnarray*}
I_3 &\leq &
C_5\mu_k^{m-1}\int_{\left\{x:|x|>\mu_k+\delta\mu_k^{-\frac{2m-1}{3}}\right\}}
{(x^{2m}-\mu_k^{2m})^{-\frac12}
\exp\left\{-\frac{\sqrt{\delta}}{3}\mu_k^{\frac{2m-1}{3}}(|x|-\mu_k)\right\}\,\mathrm dx}\\
&\leq &C_5\mu_k^{-\frac12} \int_{\mu_k+\delta\mu_k^{-\frac{2m-1}{3}}}^{\infty}{(x-\mu_k)^{-\frac12}
\exp\left\{-\frac{\sqrt{\delta}}{3}\mu_k^{\frac{2m-1}{3}}(x-\mu_k)\right\}\,\mathrm dx}.
\end{eqnarray*}
By changing variable $\mu_k^{\frac{2m-1}{3}}(x-\mu_k)=y$ in the second integral above, we
can find
 $k_3\in\mathbb{N}^+$ and $C_6>0$ so that
\begin{eqnarray}\label{equ3.65}
I_3\leq C_6\mu_k^{-\frac12-\frac{2m-1}{6}}\leq \frac{C_1}{100}\;\;\mbox{for all}\;\;k\ge k_3.
\end{eqnarray}
Finally, it follows from \eqref{equ-10-5-1710}, \eqref{equ3.62}, \eqref{equ3.63}, \eqref{equ3.64}
and \eqref{equ3.65} that there is  $C>0$, independent of $k$, so that
\begin{eqnarray}\label{equ3.65-1801}
\frac{|E\bigcap\{x:|x|<\rho_0\mu_k\}|}{\mu_k}\geq C,\;\;\mbox{when}\;\;k\ge k_4:=\max\{k_1,\,k_2,\,k_3\}.
\end{eqnarray}
By \eqref{equ3.65-1801} and  Lemma \ref{lemma4.1-4-17}, we deduce that
\begin{align}\label{equ-2-24-2}
\varliminf_{\mathbb{N}\ni k\rightarrow  \infty} \frac{|E\bigcap[-\rho_0b k^{\frac{1}{m+1}},\, \rho_0bk^{\frac{1}{m+1}}]|}{\rho_0bk^{\frac{1}{m+1}}}>0,
\end{align}
where $b$ is given by \eqref{4.4-10-17}. By \eqref{equ-2-24-2} and  Proposition \ref{prop-915}, we conclude that
  $E$ is \emph{weakly thick}.

In summary, we have completed the proof of Theorem \ref{thm-4-ne}.\qed

\begin{remark}\label{rmk-bound-set}
\noindent $(i)$  It follows immediately from  Theorem  \ref{thm-4-ne} that  any bounded set in $\R$ is not an observable  set at any time for  \eqref{equ-1.0}
 (with $m\in \mathbb{N}^+$). Indeed, if $E\subset [-R, R]$ for some $R>0$, then
\begin{align}\label{equ-1026-10}
 \frac{|E\bigcap [-x,\, x]|}{|x|} \leq \frac{2R}{|x|}\rightarrow 0,\,\,\,\text{as}\,\,x \rightarrow +\infty.
\end{align}
Thus $E$ is not  weakly thick, consequently, it is not  an observable set for  \eqref{equ-1.0}.

 As a comparison, it is natural to ask if a bounded measurable subset $E\subset\R$ is an  observable set  for the heat equation:
 $\partial_t u+Hu=0$ where $H=-\Delta+|x|^{2m}$ with $m> 1$.  This seems to be open (see \cite{M09}).

 \noindent $(ii)$ One can also construct unbounded sets which are not observable sets at any time for  \eqref{equ-1.0}
 (with $m\in \mathbb{N}^+$).
 For example, let
\begin{align}\label{equ3.61}
E:=\bigcup_{j=1}^{\infty}{E_j},\;\;\mbox{with}\;\;E_j:=[j,\,j+(j+1)^{-\varepsilon}],\,\,\,\varepsilon>0.
\end{align}
It is clear that $E\subset [0, \infty)$; $E$ is unbounded; $|E|=\infty$ if $0<\varepsilon\leq 1$. Let $x>2$ and let $j_0$ be the unique positive integer so that  $j_0\leq x<j_0+1$. Then by \eqref{equ3.61}  we have
\begin{align}\label{equ3.61-1801}
 \frac{|E\bigcap [-x, x]|}{|x|}
 &\leq\frac{1}{j_0}\sum_{j=1}^{j_0}{(j+1)^{-\varepsilon}} \leq Cx^{-\varepsilon}.
\end{align}
Since the right hand side  of \eqref{equ3.61-1801} tends to $0$ as $x \to +\infty$,  $E$ is not  weakly thick, therefore it is not an observable set for  \eqref{equ-1.0} for all $m\in \mathbb{N}^+$.
 \end{remark}

\subsection{Comparison of thicknesses for two kinds of observable sets}\label{sec4-diff}
 In this subsection, we shall show that the class of \emph{thick} sets is strictly included in the class of \emph{weakly thick sets}.

\begin{proposition}\label{prop-916}
Every  thick  set is   weakly thick.
\end{proposition}
\begin{proof}
Let $E$ be a \emph{thick set}. According to the definition ($\textbf{D}_3$), there is $L>0$ and $\gamma>0$ so that
\begin{align}\label{equ-915-1}
\frac{|E\bigcap [x,x+L]|}{L}\geq \gamma\;\;\mbox{ for all }\;\;   x\in \mathbb{R}.
\end{align}
Given $x\geq   L$, there is $n_x\in \mathbb{N}^+$ so that
\begin{align}\label{equ-2-23-11}
n_xL\leq x < (n_x+1)L.
\end{align}
By \eqref{equ-2-23-11}, we find
\begin{eqnarray*}
 [-x,x] \supset [-n_x L,n_x L] = \bigcup_{-n_x \leq j \leq n_x-1} [j L, (j+1)L].
\end{eqnarray*}
This, along with  \eqref{equ-915-1} and \eqref{equ-2-23-11}, yields
\begin{eqnarray*}
\Big|E\bigcap[-x,x]\Big|
&\geq & \sum_{-n_x\leq j \leq n_x-1} \Big|E\bigcap [j L, (j+1)L]\Big|\geq \sum_{-n_x\leq j \leq n_x-1} \gamma L\\
&=& 2\gamma n_x L \geq \gamma(n_x+1)L  > \gamma x.
\end{eqnarray*}
Since $x\geq   L$ is arbitrarily given, the above leads to
$$
\varliminf_{ x \rightarrow +\infty} \frac{|E\bigcap [-x,x]|}{x} \geq \gamma  >0.
$$
Hence,  $E$ is \emph{weakly thick}. This ends the proof of Proposition \ref{prop-916}.
\end{proof}

\begin{proposition}\label{prop-916-1}
A  weakly thick  set may not  be  thick.
\end{proposition}
This can be seen from the next two  examples.

\begin{example}\label{exp-1}
Let $E=[a,\infty)$ with $a\in \mathbb{R}$. Then $E$ is weakly thick but not thick.
\end{example}
Here is the proof: Arbitrarily fix $a\in \mathbb{R}$ and let $E=[a,\infty)$. Then, when $x>|a|$, we have $|E\cap[-x,x]|=x-a$. Hence,
$$
\varliminf_{ x \rightarrow +\infty} \frac{|E\bigcap [-x,x]|}{x}=1.
$$
So $E$ is \emph{weakly thick}.

On the other hand, we arbitrarily fix $L>0$. Then we have that for each
$x<a-L$, $E\cap[x,x+L]=\emptyset$.
 So we have
$$
|E\cap[x,x+L]|=0\;\;\mbox{for all}\;\; x<a-L.
$$
This shows that $E$ is not \emph{thick}.

\begin{example}\label{exp-2}
Let \begin{align}\label{equ-915-4}
E=\bigcup_{j=1}^\infty \Big([2^j, 2^{j+1}-j]\cup [-2^{j+1}+j,-2^j]\Big).
\end{align}
Then $E$ is weakly thick but not thick.
\end{example}
Here is the proof:  We first show that  $E$ is \emph{weakly thick}.   To this end, we first claim that
\begin{align}\label{equ-915-5}
\varliminf_{  x \rightarrow +\infty} \frac{|E\bigcap [0,x]|}{x}  >0.
\end{align}
For this purpose, we arbitrarily fix $x\geq 32$. Then there is a unique integer $n\geq 5$
so that
\begin{align}\label{4.39,10-24}
2^n\leq x<2^{n+1}.
\end{align}
From \eqref{4.39,10-24},
 we find that $[0,x]\supset [0,2^{n}]\supset [2^{n-1},2^n-n+1]$.  This, along with  \eqref{equ-915-4}, yields that
\begin{align*}
E\cap [0,x] \supset [2^{n-1},2^n-n+1].
\end{align*}
The above, together with  \eqref{4.39,10-24}, implies that
\begin{eqnarray*}
\Big|E\cap [0,x]\Big| \geq 2^{n-1}-n+1>x/8,
\end{eqnarray*}
which leads to \eqref{equ-915-5}. By \eqref{equ-915-5}, we find that $E$ is \emph{weakly thick}.

We next  show that $E$ is not \emph{thick}. To this end, we arbitrarily fix  $L>0$. Let $k\in \mathbb{N}^+$ so that $k> 3L$. Using \eqref{equ-915-4}, we deduce that
\begin{align}\label{equ-2-23-12}
E \bigcap [2^k-3L,2^k-L]=\emptyset.
\end{align}
But clearly  one has
 \begin{align}\label{equ-2-23-13}
[x,x+L]\subset [2^k-3L,2^k-L],\;\;\mbox{when}\;\; x\in [2^k-3L,2^k-2L].
\end{align}
Combining \eqref{equ-2-23-12} and \eqref{equ-2-23-13}, we find that
$$
 E \bigcap [x,x+L] = \emptyset,\;\;\mbox{when}\;\; x\in [2^k-3L,2^k-2L].
$$
This shows that $E$ is not \emph{thick}.

\begin{remark}\label{rmk-4.13}
Given a  sequence $\{a_j\}_{j=1}^{\infty}$, with $0<a_j<2^{j-1}$, let
 \begin{align}\label{equ-915-4-1}
\tilde{E} = \bigcup_{j=1}^\infty \Big([2^j, 2^{j+1}-a_j]\bigcup [-2^{j+1}+a_j,-2^j]\Big).
\end{align}
Then  $\tilde{E}$ is \emph{ thick } if and only if $\{a_j\}_{j=1}^{\infty}$ is bounded, i.e., there is $L_0>0$ (independent of $j$) so that
 \begin{align}\label{equ-915-4-2}
a_j\leq L_0   \;\;\mbox{for all}\;\;j\ge 1.
\end{align}

Indeed, if $\tilde{E}$ is \emph{thick}, then by \eqref{equ-set} (where $E=\tilde{E}$ and
$x=2^{j+1}-a_j$),  we see that for all $j$ large enough,
$$
\big|[2^{j+1}, 2^{j+1}-a_j+L]\big|=\big|\tilde{E}\bigcap [x,x+L]\big|\ge L\gamma>0.
$$
Since  $a_j<2^{j-1}$, the above leads to  \eqref{equ-915-4-2}.
 Conversely, we suppose that \eqref{equ-915-4-2} is true. Set $g(x):=|\tilde{E}\bigcap [x,x+2L_0]|/(2L_0)$. By \eqref{equ-915-4-1} and \eqref{equ-915-4-2}, we can find
  $M>0$ so that $g(x)\ge \frac12$, when $|x|>M$. Since $g(\cdot)$ is continuous and positive
  over $[-M,M]$, we can choose  $\gamma_0>0$ so that $g(x)\ge \gamma_0$ for all $x\in\mathbb{R}$. So $\tilde{E}$ is thick.

Intuitively speaking, the sequence  $\{a_j\}_{j=1}^{\infty}$ describes the gaps of the set $\tilde{E}$. From \eqref{equ-915-4-2}, we see that a thick set must have uniformly bounded gaps, while from \eqref{equ-915-4} ($a_j=j$), we find that a weakly thick set can contain increasing gaps of arbitrarily large size.

\end{remark}

\section{Further results for Hermite Schr\"{o}dinger equations in
$\mathbb{R}^n$}\label{two-point-hermite}
We start with recalling several known facts related to the spectral theory of the Hermite operator $H=-\Delta+|x|^2$ (in $L^2(\mathbb{R}^n)$), which can be found in \cite{Ti,Th}. {\it The first one} is about eigenvalues:
\begin{align}\label{equ1.3}
\sigma(H)=\{n+2k,\,\,\,k=0,1,2,\ldots.\}.
\end{align}
{\it The second one} is about  eigenfunctions of $H$: For each $k\in \mathbb{N} $, let
\begin{align}\label{equ1.4}
\varphi_k(x)=(2^kk!\sqrt{\pi})^{-\frac12}H_{k}(x) e^{-\frac{x^2}{2}},\;x\in \mathbb{R},
\end{align}
where  $H_{k}$ is the Hermite polynomial  given by
\begin{eqnarray}\label{5.5,11-20}
H_{k}(x)=(-1)^k e^{x^2}\frac{d^k}{dx^k}(e^{-x^2}), \; x\in \mathbb{R}.
\end{eqnarray}
Notice that $\|\varphi_k(x)\|_{L^2(\mathbb{R})}=1$ for all $k\in \mathbb{N}$.
Now for each multi-index $\alpha=(\alpha_1, \alpha_2\ldots \alpha_n)$ ($\alpha_i\in \mathbb{N}$), we define the following $n$-dimensional Hermite function by tensor product:
\begin{eqnarray}\label{5.6,10-27}
\Phi_{\alpha}(x)=\Pi_{i=1}^n{\varphi_{\alpha_i}(x_i)},\;\;x=(x_1,\dots,x_n)\in \mathbb{R}^n.
\end{eqnarray}
Then for each $k\in \mathbb{N}$, $\Phi_{\alpha}$ (with $|\alpha|=k$) is an eigenfunction of $H$ corresponding to the eigenvalue $n+2k$, and
$\{\Phi_{\alpha}\;:\;\alpha\in \mathbb{N}^n\}$ forms a complete orthonormal
basis in $L^2(\mathbb{R}^n)$.
{\it The third one} is about  the solution $u$ of the Hermite-Schr\"{o}dinger equation \eqref{equ1.1}
with the initial condition $u(0, \cdot)=f(\cdot)\in L^2(\mathbb{R}^n)$:
\begin{align}\label{equ1.7}
u(t,x)=e^{-\i tH}f=\sum_{\alpha\in \mathbb{N}^n}{e^{-\i t(n+2|\alpha|)}a_k\Phi_{\alpha}}(x),
\;\;t\geq 0,\;x\in \mathbb{R}^n,
\end{align}
where $a_k=\int_{\mathbb{R}^n}{f(x)\Phi_{\alpha}(x)\,\mathrm dx}$ is the Fourier-Hermite coefficient. Let  $K(t, x, y)$ be the kernel associated to the operator $e^{-\i tH}$. Then by Mehler's formula (see e.g. in \cite{ST,Th}), we have
\begin{align}\label{wang5.6}
e^{-\i tH}f=\int_{\mathbb{R}^n}{K(t, x, y)f(y)\,\mathrm dy},\,\,\,\,t\in\mathbb{R}^+\setminus \frac{\pi}{2}\mathbb{N},
\end{align}
where
\begin{align}\label{equ1.12}
K(t, x, y)=\frac{e^{-\i \pi n/4}}{(2\pi\sin{2t})^{n/2}}\exp\left(\frac{\i}{2}(|x|^2+|y|^2)\cot{2t}-\frac{\i}{\sin{2t}}x\cdot y\right),\,x, y\in \mathbb{R}^n.
\end{align}
Meanwhile, it follows by \eqref{equ1.7} that
\begin{align}\label{equ1.8}
\|e^{-\i tH}f\|_{L^2(\mathbb{R}^n)}=\|f\|_{L^2(\mathbb{R}^n)}\;\;\mbox{for all}\;\;t\geq 0
\end{align}
and that
\begin{align}\label{equ1.9}
e^{-\i (t+\pi)H}f=e^{-i\pi n}e^{-\i tH}f\;\;\mbox{for all}\;\;t\geq 0.
\end{align}

\subsection{Proof of Theorem \ref{thm-1126-2}.}\label{subsec5.1} As already mentioned in Remark ($\textbf{c}_2$) in the introduction, we first build up  the observability inequality at two points in time  for the equation \eqref{equ1.1}, then by using it, obtain the observability inequality \eqref{equ-913-0} for any $T>0$.
Since we are in the general case where $n\geq 1$,
  the spectral approach used to prove Theorem \ref{thm-ob-for-free} and Theorem \ref{thm-ob-for-HO}
  seem not work. (At least, we do not know how to use it.)
  Fortunately, the kernel, associated with  $e^{-\i tH}$, has an explicit expression given by
  \eqref{equ1.12}. This expression can help us to look at the problem from a new perspective. In particular, we realize some connections between  uncertainty principles in harmonic analysis and observability inequalities.
   It deserves mentioning what follows: $(i)$ The aforementioned observability inequality at two time points
  was obtained in  \cite{WWZ} for the free Schr\"{o}dinger equation; $(ii)$  {In \cite{HS}, the authors considered a class of decaying potentials $V$ and established observability inequality at two points in time for $H=-\Delta+V$. To our best knowledge, no such kind of results have been proved for potentials that are increasing to infinity when $|x|\rightarrow\infty$.
}

\begin{theorem}\label{thm-two-point-ob}
The following conclusions are true:

(\romannumeral1) If $T$ and $S$, with $T>S\ge 0$, satisfy that $T-S\ne \frac{k\pi}{2}$ for all $k\in \mathbb{N}^+$, then there is $C:= C(n)$ so that
\begin{align}\label{equ1.10}
\int_{\mathbb{R}^n}{|u(0, x)|^2\,\mathrm dx}\leq Ce^{\frac{Cr_1 r_2}{\sin{2(T-S)}}}\left(\int_{B^c(x_1,\, r_1)}{|u(S, x)|^2\,\mathrm dx}+\int_{B^c(x_2,\, r_2)}{|u(T, x)|^2\,\mathrm dx}\right)
\end{align}
holds for any closed balls $B(x_1,\, r_1)$ and $B(x_2,\, r_2)$ in $\mathbb{R}^n$ and any solution $u$ to  \eqref{equ1.1}.

(\romannumeral2) If $T$ and $S$, with $T>S\ge 0$, satisfy that $T-S=\frac{k\pi}{2}$ for some $k\in \mathbb{N}^+$, then for any closed balls $B(x_1,\, r_1)$ and $B(x_2,\, r_2)$ in $\mathbb{R}^n$,  there is no $C>0$ so that
\begin{align}\label{equ1.11}
\int_{\mathbb{R}^n}{|u(0, x)|^2\,\mathrm dx}\leq C\left(\int_{B^c(x_1,\, r_1)}{|u(S, x)|^2\,\mathrm dx}+\int_{B^c(x_2,\, r_2)}{|u(T, x)|^2\,\mathrm dx}\right)
\end{align}
holds for all solutions $u$ to  \eqref{equ1.1}.
\end{theorem}
\begin{proof}
Let  $K(t, x, y)$ be the kernel associated to  $e^{-\i tH}$.
The following two facts are needed. First, when $t\in\mathbb{R}^+\setminus \frac{\pi}{2}\mathbb{N}$, $K(t, x, y)$ is given by \eqref{equ1.12}.
 (Notice that the structure of the above kernel  breaks down and becomes singular at resonant times $t=\frac{\pi}{2}\cdot k$, $k=0,1,2,\ldots$.)
Second,  one has (see e.g. in \cite{KRY})
\begin{align}\label{equ1.13}
K\left(\frac{k\pi}{2}, x, y\right)=e^{-\frac{\i k\pi n}{2}}\delta\left(x-(-1)^ky\right), \,x, y\in \mathbb{R}^n.
\end{align}
where $\delta$ is the Dirac  function.

We will use \eqref{equ1.12} to show  the conclusion $(i)$.
Arbitrarily fix a solution $u$ to   \eqref{equ1.1}
 and two balls $B(x_1,\, r_1)$ and $B(x_2,\, r_2)$.
By \eqref{equ1.8} and \eqref{equ1.9}, we can assume, without loss of generality, that $S=0$ and $0<T\leq \pi$.
 In this case, we have $0<T<\pi$ and $T\ne \frac{\pi}{2}$, which implies that $\sin{2T}\ne 0$. The key observation is  as:
\begin{align}\label{equ1.14}
u(T, x)&=\frac{e^{-\i \pi n/4}}{(2\pi\sin{2T})^{n/2}}\int_{\mathbb{R}^n}
{\exp\left(\frac{\i}{2}(|x|^2+|y|^2)\cot{2T}-\frac{\i}{\sin{2T}}x\cdot y\right)u(0, y)\,\mathrm dy}\nonumber\\
&=\frac{e^{-\i \pi n/4}}{(2\pi\sin{2T})^{n/2}}e^{\frac{\i |x|^2}{2}\cdot \cot{2T}}\int_{\mathbb{R}^n}{e^{-\i \frac{x}{\sin{2T}}\cdot y}u(0, y)e^{\i \frac{|y|^2}{2}\cdot \cot{2T}}\,\mathrm dy}\nonumber\\
&=\frac{e^{-\i \pi n/4}}{(2\pi\sin{2T})^{n/2}}e^{\frac{\i |x|^2}{2}\cdot \cot{2T}}\mathcal{F}\left(e^{\i \frac{|\cdot|^2}{2}\cdot \cot{2T}}u(0, \cdot)\right)\left(\frac{x}{\sin{2T}}\right),\; x\in \mathbb{R}^n,
\end{align}
where $\mathcal{F}$ stands for the Fourier transform.
Recall that the  uncertainty principle built up in \cite{Jam} says: {\it for any $S, \Sigma\subset \mathbb{R}^n$ with $|S|<\infty$ and $|\Sigma|<\infty$, there is a positive constant
\begin{eqnarray}\label{equ1.15}
C(n, S, \Sigma):=Ce^{C\min\{|S||\Sigma|,\,\, |S|^{1/n}\omega(\Sigma),\,\, |\Sigma|^{1/n}\omega(S)\}},
 \end{eqnarray}
with $C=C(n)$, so that for any $g\in L^2(\mathbb{R}^n)$,
\begin{align}\label{equ-Nazarov-UP}
\int_{\mathbb{R}^n}{|g|^2\,\mathrm dx}\leq C(n, S, \Sigma)\left(\int_{S^c}{|g|^2\,\mathrm dx}+\int_{\Sigma^c}{|\hat{g}|^2\,\mathrm dx}\right).
\end{align}
(Here $\omega(S)$ denotes the mean width of $S$, we refer the readers to \cite{Jam} for
 its detailed definition. In particular, when $S$ is a ball in $\mathbb{R}^n$, $\omega(S)$ is the diameter of the ball.)}

 By \eqref{equ-Nazarov-UP}, where  $g(x)=e^{\i \frac{|x|^2}{2}\cdot \cot{2T}}u(0, x)$  and
 $(S^c,\, \Sigma^c)$  is replaced by $\Big(B^c(x_1,\, r_1),\, \frac{B^c(x_2,\, r_2)}{\sin{2T}}\Big)$ (here we have used the notation $kE:=\{kx,\,x\in E\}$) and then by \eqref{equ1.14}, we  find
\begin{align}\label{equ1.17}
\int_{\mathbb{R}^n}{|u(0, x)|^2\,\mathrm dx}\leq C\left(\int_{B^c(x_1,\, r_1)}{|u(0, x)|^2\,\mathrm dx}+\int_{B^c(x_2,\, r_2)}{|u(T, x)|^2\,\mathrm dx}\right),
\end{align}
where $C:=C\left(n, B(x_1,\, r_1), \frac{B(x_2,\, r_2)}{\sin{2T}}\right)$ is given by \eqref{equ1.15}.
In view of \eqref{equ1.15}, we find that
\begin{align}\label{equ1.18}
C\left(n, B(x_1,\, r_1), \frac{B(x_2,\, r_2)}{\sin{2T}}\right)\leq Ce^{\frac{C r_1 r_2}{\sin{2T}}},
\end{align}
which, along with \eqref{equ1.17}, leads to \eqref{equ1.10}.

Next, we will use  \eqref{equ1.13} to prove the conclusion $(\romannumeral2)$. Without loss of generality,
 we can assume  that $(S,T)=(0, \frac{\pi}{2})$ or $(S,T)=(0,\pi)$.
 In the case when $(S,T)=(0, {\pi})$, we see from
  \eqref{equ1.13} that $K(\pi, x, y)=e^{-\i \pi n}\delta(x-y)$, which implies that for any solution $u$ to \eqref{equ1.1},
\begin{align}\label{equ1.19}
|u(0, x)|=|u(\pi, x)|,\,\,\,\,x\in\mathbb{R}^n.
\end{align}
To simplify matters, we set $x_1=x_2=0$.
Let $f$ be a nonzero function in $C_0^{\infty}(B(0, r))$, with $r=\min\{r_1, r_2\}$.
Let $v$ be the solution to \eqref{equ1.1} with the initial condition: $v(\cdot,0)=f(\cdot)$.
Then, for this solution $v$,
  the left hand side of the inequality \eqref{equ1.11} is strictly positive, but the right hand side of \eqref{equ1.11} is zero since both integrals vanish.
   (Here we used \eqref{equ1.19}.) This shows that for this solution $v$, \eqref{equ1.11} is not true in the case that $(S,T)=(0, {\pi})$.

We now consider the case that  $(S,T)=(0, \frac{\pi}{2})$.
 To simplify matters, we again set $x_1=x_2=0$. Let
 $u_0(\cdot)\in C_0^{\infty}((-r/\sqrt{n}, r/\sqrt{n}))$ (with $r=\min\{r_1, r_2\}$) be a nonzero real-valued  even
 function (i.e., $u_0(x)=u_0(-x)$ for all $x\in \mathbb{R}$). Then define a function by
\begin{align}\label{equ1.20-1118}
g(x):=\prod_{i=1}^n{u_0(x_i)},\;\; x=(x_1,\dots,x_n)\in \mathbb{R}^n.
\end{align}
Let $w$ be the solution to \eqref{equ1.1} with the initial condition: {$w(0, \cdot)=g(\cdot)$, where $g$ is given by \eqref{equ1.20-1118}.
It is clear that
 $\text{supp}\,  w(0, x)\subset B(0, r)$.}{}
We claim
\begin{align}\label{equ1.20}
|w(0, x)|=\left|w\left({\pi}/{2}, x\right)\right|,\,\,\,\,x\in\mathbb{R}^n.
\end{align}
Indeed,
since $\{\Phi_{\alpha} : \alpha\in \mathbb{N}^n\}$ forms a complete orthonormal
basis in $L^2(\mathbb{R}^n)$, we have
\begin{align}\label{equ1.21.1129}
w(0, x)=\sum_{\alpha\in\mathbb{N}^n}\langle w(0, \cdot),\, \Phi_{\alpha}(\cdot) \rangle_{L^2(\mathbb{R}^n)} \Phi_{\alpha}.
\end{align}
Meanwhile, by \eqref{equ1.4}, \eqref{5.5,11-20} and \eqref{5.6,10-27}, we see that
when  $\alpha_i$ is odd/even, $\varphi_{\alpha_i}$ is odd/even.
Thus,  when
$|\alpha|$ is odd, there exists some $j$, $1\leq j\leq n$, such that $\alpha_j$ is odd, consequently,  $\varphi_{\alpha_j}$ is an odd function, which implies that $\langle u_0(x_j),\,\varphi_{\alpha_j}(x_j) \rangle_{L^2(\mathbb{R})}=0$. This, along with \eqref{equ1.20-1118}, yields
\begin{align}\label{equ1.21}
\langle w(0, \cdot),\, \Phi_{\alpha}(\cdot) \rangle_{L^2(\mathbb{R}^n)} =\prod_{i=1}^n\langle u_0(\cdot),\,\varphi_{\alpha_i}(\cdot) \rangle_{L^2(\mathbb{R})}=0,\,\,\,\text{if}\,\, |\alpha|\,\, \text{is odd}.
\end{align}
From \eqref{equ1.21.1129} and \eqref{equ1.21}, we see
\begin{align}\label{equ1.22}
w(0, \cdot)=\sum_{|\alpha|\, \text{is even}}\langle w(0, \cdot),\, \Phi_{\alpha}(\cdot) \rangle_{L^2(\mathbb{R}^n)} \Phi_{\alpha}(x),\; x\in \mathbb{R}^n.
\end{align}
On the other hand, we obtain from \eqref{equ1.7} that
\begin{align}
w({\pi}/{2}, x)&=\sum_{\alpha\in\mathbb{N}^n}e^{-\i \frac{\pi}{2}(n+2|\alpha|)}\langle w(0, \cdot),\, \Phi_{\alpha}(\cdot) \rangle_{L^2(\mathbb{R}^n)} \Phi_{\alpha}(x)\nonumber\\
&=e^{-\frac{\i n\pi}{2}}\sum_{\alpha\in\mathbb{N}^n}(-1)^{|\alpha|}\langle w(0, \cdot),\, \Phi_{\alpha}(\cdot) \rangle_{L^2(\mathbb{R}^n)} \Phi_{\alpha}(x)\nonumber\\
&=e^{-\frac{\i n\pi}{2}}\sum_{|\alpha|\, \text{is even}}\langle w(0, \cdot),\, \Phi_{\alpha}(\cdot) \rangle_{L^2(\mathbb{R}^n)} \Phi_{\alpha}(x),\;
x\in \mathbb{R}^n,\nonumber
\end{align}
which, together with  \eqref{equ1.22}, leads to  \eqref{equ1.20}.

Hence, for the above solution $w$,  the left hand side of the inequality \eqref{equ1.11} is strictly positive, but the right hand side of \eqref{equ1.11} is zero. (Here we used \eqref{equ1.20}.)
This ends the proof of Theorem \ref{thm-two-point-ob}.
\end{proof}

%
Based on Theorem \ref{thm-two-point-ob}, we are on the position to  show  Theorem \ref{thm-1126-2}.

\textbf{Proof of Theorem \ref{thm-1126-2}.}   Arbitrarily fix a ball $B(x_0,r)$ and a solution  $u$ to \eqref{equ1.1}. We first consider the case when $0<T\leq\pi/4$.
According to   $(i)$ of Theorem \ref{thm-two-point-ob} (with $r_1=r_2=r$ and $x_1=x_2=x_0$), there exists $C=C(n)$ so that when $0\leq s<t<T$,
\begin{align}\label{equ-913-1}
\int_{\mathbb{R}^n}{|u(0, x)|^2\,\mathrm dx}\leq Ce^{\frac{Cr^2}{\sin{2(t-s)}}}\left(\int_{B^c(x_0,\, r)}{|u(s, x)|^2\,\mathrm dx}+\int_{B^c(x_0,\, r)}{|u(t, x)|^2\,\mathrm dx}\right).
\end{align}
Since $0\leq s<t<T$ and $T<\pi/4$, we have $0<2(t-s)<\pi/2$. Thus $\sin 2(t-s) \geq 4(t-s)/\pi$. Moreover, if $(s,t)\in [0,T/3]\times [2T/3,T]$, we have $2(t-s)\geq 2T/3$. These show that
\begin{align}\label{equ-913-2}
\sin{2(t-s)} \geq 4T/(3\pi), \quad (s,t)\in [0,T/3]\times [2T/3,T].
\end{align}
Integrating \eqref{equ-913-1} with $s$ over $s\in [0,T/3]$ and $t$ over $t\in [2T/3,T]$, using \eqref{equ-913-2}, we obtain that
\begin{eqnarray*}
\left(\frac{T}{3}\right)^2\int_{\mathbb{R}^n}{|u(0, x)|^2\,\mathrm dx}
&\leq& C\int_0^{\frac{T}{3}}\int_{\frac{2T}{3}}^Te^{\frac{Cr^2}{4T/(3\pi)}}\left(\int_{B^c(x_0,\, r)}{|u(s, x)|^2\,dx}+\int_{B^c(x_0,\, r)}{|u(t, x)|^2\,dx}\right)\mathrm dt \mathrm ds\nonumber\\
&\leq& \frac{C T}{3}e^{\frac{3\pi Cr^2}{4T}}\left(\int_0^{\frac{T}{3}}\int_{B^c(x_0,\, r)}|u(s, x)|^2\,dx \d s+ \int_{\frac{2T}{3}}^T\int_{B^c(x_0,\, r)}|u(t, x)|^2\,\mathrm dx \mathrm dt\right)\nonumber\\
&\leq& \frac{C T}{3}e^{\frac{3\pi Cr^2}{4T}} \int_0^T\int_{B^c(x_0,\, r)}|u(t, x)|^2\,\mathrm dx \mathrm dt.
\end{eqnarray*}
From the above, we obtain
\begin{align}\label{equ-913-4}
 \int_{\mathbb{R}^n}{|u(0, x)|^2\,\mathrm dx}
 \leq\frac{3 C}{T}e^{\frac{3\pi Cr^2}{4T}} \int_0^T\int_{B^c(x_0,\, r)}|u(t, x)|^2\,\mathrm dx \mathrm dt,
\end{align}
which leads to  \eqref{equ-913-0} for the case that $0<T\leq \pi/4$.

We next consider the case when  $T>\pi/4$. By \eqref{equ-913-4} with $T=\pi/4$, we find
\begin{eqnarray*}
 \int_{\mathbb{R}^n}{|u(0, x)|^2\,\mathrm dx}
 \leq\frac{12 C}{\pi}e^{3Cr^2} \int_0^{\frac{\pi}{4}}\int_{B^c(x_0,\, r)}|u(t, x)|^2\,\mathrm dx \mathrm dt,
\end{eqnarray*}
from which, it follows that when $T>\pi/4$,
\begin{eqnarray*}
 \int_{\mathbb{R}^n}{|u(0, x)|^2\,\mathrm dx}
 \leq\frac{12 C}{\pi}e^{\frac{6Cr^2}{\pi}} \int_0^{T}\int_{B^c(x_0,\, r)}|u(t, x)|^2\,\mathrm dx \mathrm dt.
\end{eqnarray*}
The above leads to \eqref{equ-913-0} for the case where $T> \pi/4$. This ends the proof of
Theorem \ref{thm-1126-2}.\qed

\subsection{Proof of Theorem \ref{thm-1126-1}.}\label{subsec5.2}
Since the equation \eqref{equ1.1} is rotation invariant, we can assume, without loss of generality, that $a=(1,\ldots, 0)\in\mathbb{R}^n$. In what follows, $E:=B^c(0,\, r)\bigcap\{x\in\mathbb{R}^n:\, x_1\ge 0\}$.

\noindent{\it Step 1.  We  prove that $(i)\Rightarrow (ii)$.}

Given $k\in \mathbb{N}^+$, write $\vec{k}:=(k,\ldots, k)\in \mathbb{R}^n$; let
\begin{align}\label{wang5.26}
u_{0,\,k}(x):=\pi^{-\frac{n}{4}}e^{-\frac{|x|^2}{2}-i\vec{k}\cdot x},\;\;x\in \mathbb{R}^n;
\end{align}
and write $u_k(t,x)$ for the solution of \eqref{equ1.1}, with the initial condition:
 $u_k(0,x)=u_{0,\,k}(x)$.
We claim that
\begin{align}\label{1120-2}
\|u_{0,\,k}\|_{L^2(\mathbb{R}^n)}=1\;\;\mbox{for all}\;\;k\in \mathbb{N}^+
\end{align}
and  that
 \begin{align}\label{1120-3}
\lim_{k\rightarrow\infty}\int_0^{T}\int_{E}|u_k(t,x)|^2\,\mathrm dx\,\mathrm dt=0
\;\;\mbox{for each}\;\; T\in (0,  \frac{\pi}{2}].
\end{align}
When this is done, ``$(i)\Rightarrow (ii)$" follows from \eqref{1120-2}
and \eqref{1120-3} at once.

The equality \eqref{1120-2} follows from \eqref{wang5.26} and the direct calculation:
$$
\int_{\mathbb{R}^n}{|u_{0,\,k}|^2\,\mathrm dx}=\pi^{-\frac{n}{2}}\int_{\mathbb{R}^n}{e^{-|x|^2}\,\mathrm dx}=1\;\;\mbox{for each}\;\;
k\in \mathbb{N}^+.
$$

 We next show \eqref{1120-3}. By  \eqref{wang5.6} and \eqref{wang5.26}, we have
\begin{align*}
u_k(t, x)&=\frac{e^{-\i n\pi/4}}{(2\pi\sin{2t})^{n/2}}\int_{\mathbb{R}^n}{\exp\left(\frac{\i}{2}(|x|^2+|y|^2)\cot{2t}-\frac{\i}{\sin{2t}}x\cdot y\right)u_{0,\,k}(y)\,\mathrm dy}\nonumber\\
&=\frac{e^{-\i n\pi/4}}{(2\pi\sin{2t})^{n/2}}e^{\frac{\i |x|^2}{2}\cdot \cot{2t}}\int_{\mathbb{R}^n}{e^{-\i \frac{x}{\sin{2t}}\cdot y}u_{0,\,k}(y)e^{\i \frac{|y|^2}{2}\cdot \cot{2t}}\,\mathrm dy},\,\,t\in (0,{\pi}/{2}), x\in \mathbb{R}^n.
\end{align*}
By changing variables in the above, we find
\begin{align*}
u_k(t, x)&=\frac{e^{-\i n \pi /4}}{(2\pi\sin{2t}\cdot A_t)^{n/2}}\exp\left\{\frac{\i |x|^2}{2}\cdot\cot{2t}-\frac{|\frac{x}{\sin{2t}}+\vec{k}|^2}{2A_t}\right\},
\,\,t\in (0,{\pi}/{2}), x\in \mathbb{R}^n,
\end{align*}
where $A_t:=1-\i\cot{2t}$.
This implies that when  $0<t<\frac{\pi}{2}$ and $x\in \mathbb{R}^n$,
\begin{align}\label{1120-6}
|u_k(t, x)|&\leq C\exp\left\{-\frac{|\frac{x}{\sin{2t}}+\vec{k}|^2}{2(1+\cot^2{2t})}\right\}= C\exp\left\{-\frac{|x+\vec{k}\sin{2t}|^2}{2}\right\}.
\end{align}

Now we arbitrarily fix  $0<\epsilon <\frac{\pi}{4}$. Several facts are given in order. Fact One:
\begin{align}\label{1120-6-01}
\int_0^{\frac{\pi}{2}}\int_{E}|u_k(t,x)|^2\,\mathrm dx\,\mathrm dt\leq\int_0^{\frac{\epsilon}{3}}\int_{\mathbb{R}^n}|u_k(t,x)|^2\,\mathrm dx\,\mathrm dt&+\int_{\frac{\epsilon}{3}}^{\frac{\pi}{2}-\frac{\epsilon}{3}}\int_{E}|u_k(t,x)|^2\,\mathrm dx\,\mathrm dt\nonumber\\
&+\int_{\frac{\pi}{2}-\frac{\epsilon}{3}}^{\frac{\pi}{2}}\int_{\mathbb{R}^n}|u_k(t,x)|^2\,\mathrm dx\,\mathrm dt.
\end{align}
Fact Two: It follows directly from \eqref{1120-2} that
\begin{align}\label{1120-6-02}
 \int_0^{\frac{\epsilon}{3}}\int_{\mathbb{R}^n}|u_k(t,x)|^2\,\mathrm dx\,\mathrm dt+\int_{\frac{\pi}{2}-\frac{\epsilon}{3}}^{\frac{\pi}{2}}\int_{\mathbb{R}^n}|u_k(t,x)|^2\,\mathrm dx\,\mathrm dt=\frac{2\epsilon}{3}.
\end{align}
Fact Three: Since $\sin{2t}\ge \sin{(2\epsilon/3)}>0$, when   $\frac{\epsilon}{3}\le t\le \frac{\pi}{2}-\frac{\epsilon}{3}$,   we deduce from \eqref{1120-6} and the definition of  $E$  that
\begin{eqnarray*}\label{1120-6-03}
\int_{E}{|u_{k}(t, x)|^2\,\mathrm dx}&\leq & C\int_{E}{e^{-|x+\vec{k}\sin{2t}|^2}\,\mathrm dx}\nonumber\\
&\leq & C\int_{k\sin{(2\epsilon/3)}}^{\infty}{e^{-x_1^2}\,\mathrm dx_1}\prod_{j=2}^{n}\int_{\mathbb{R}}{e^{-x_j^2}\,\mathrm dx_j}\longrightarrow 0,\,\,\,\text{as}\,\,k\rightarrow+\infty,
\end{eqnarray*}
from which, we can find  $K>0$ so that
\begin{align}\label{1120-6-04}
\int_{\frac{\epsilon}{3}}^{\frac{\pi}{2}-\frac{\epsilon}{3}}\int_{E}|u_k(t,x)|^2\,\mathrm dx\,\mathrm dt\leq \frac{\epsilon}{3}\;\;\mbox{for all}\;\; k>K.
\end{align}
Because $\epsilon>0$ can be arbitrarily small,  \eqref{1120-3} follows from \eqref{1120-6-01}, \eqref{1120-6-02} and \eqref{1120-6-04} immediately.

Hence, we have proved ``$(i)\Rightarrow (ii)$".

\noindent{\it Step 2. We  prove that   ``$(ii)\Rightarrow (i)$".}

Arbitrarily fix $T>\frac{\pi}{2}$. By contradiction, we suppose that $(i)$ is not true for the aforementioned $T$. Then
 there exists a sequence of functions
  $\{v_{0,\,k}\}_{k=1}^\infty$ in $L^2(\mathbb{R}^n)$ so that  for each $k\in \mathbb{N}^+$,
\begin{align}\label{0215-2}
\|v_{0,\,k}\|_2=1\;\;\mbox{and}\;\;
\int_0^{T}\int_{E}|e^{-\i tH}v_{0,\,k}|^2\,\mathrm dx\,\mathrm dt\leq \frac{\|v_{0,\,k}\|^2_{L^2}}{k}=\frac{1}{k}.
\end{align}

Several observations are given in order. First, from
\eqref{wang5.6}, we see that when
 $t\in\mathbb{R}^+\setminus \frac{\pi}{2}\mathbb{N}$ and $x\in\mathbb{R}^n$,
\begin{align}\label{0215-3}
&\left(e^{-\i (t+\pi/2)H}v_{0,\,k}\right)(x)\nonumber\\
&=\frac{e^{-\i \pi n/4}}{(2\pi\sin{2(t+\pi/2)})^{n/2}}\int_{\mathbb{R}}{\exp\left(\frac{\i}{2}(|x|^2+|y|^2)\cot{2(t+\frac{\pi}{2})}-\frac{\i}{\sin{2(t+\pi/2)}}x\cdot y\right)v_{0,k}(y)\,\mathrm dy}\nonumber\\
&=\frac{e^{-\i 3\pi n/4}}{(2\pi\sin{2t})^{n/2}}\int_{\mathbb{R}}{\exp\left(\frac{\i}{2}(|x|^2+|y|^2)\cot{2t}-\frac{\i}{\sin{2t}}(-x)\cdot y\right)v_{0,k}(y)\,\mathrm dy}\nonumber\\
&=e^{-\i \pi n/2}\left(e^{-\i tH}v_{0,\,k}\right)(-x).
\end{align}
Second, since $T>\frac{\pi}{2}$,  we have
\begin{eqnarray}\label{0215-4}
\int_0^{T}\int_{E}|e^{-\i tH}v_{0,\,k}|^2\,\mathrm dx\,\mathrm dt=\int_0^{\frac{\pi}{2}}\int_{E}|e^{-\i tH}v_{0,\,k}|^2\,\mathrm dx\,\mathrm dt+\int_\frac{\pi}{2}^{T}\int_{E}|e^{-\i tH}v_{0,\,k}|^2\,\mathrm dx\,\mathrm dt.
\end{eqnarray}
Third, by  the change of variable $t\rightarrow t+\frac{\pi}{2}$ and the fact \eqref{0215-3}, we find
\begin{align}\label{0215-5}
\int_\frac{\pi}{2}^{T}\int_{E}|e^{-\i tH}v_{0,\,k}(x)|^2\,\mathrm dx\,\mathrm dt&=\int_0^{T-\frac{\pi}{2}}\int_{E}|e^{-\i (t+\frac{\pi}{2})H}v_{0,\,k}(x)|^2\,\mathrm dx\,\mathrm dt\nonumber\\
&=\int_0^{T-\frac{\pi}{2}}\int_{-E}|e^{-\i tH}v_{0,\,k}(x)|^2\,\mathrm dx\,\mathrm dt,
\end{align}
where $-E:=B^c(0,\, r)\bigcap\{x\in\mathbb{R}^n:\, x_1\le 0\}$.
 Last, it is clear that
\begin{align}\label{0222-01}
E\bigcup(-E)=B^c(0,\, r).
\end{align}

Set  $T_0=min\{T-\frac{\pi}{2}, \frac{\pi}{2}\}$.
Then by the second inequality in \eqref{0215-2} and by \eqref{0215-4},  \eqref{0215-5} and  \eqref{0222-01}, we have
\begin{align}\label{0215-7}
\int_0^{T_0}\int_{B^c(0,\, r)}|e^{-\i tH}v_{0,\,k}|^2\,\mathrm dx\,\mathrm dt\leq\frac{1}{k}.
\end{align}
Now by Theorem \ref{thm-1126-2} (with $x_0=0$), the first equality in \eqref{0215-2}
 and  \eqref{0215-7}, we find
\begin{align}\label{0215-8}
1\leq C(T_0)\int_0^{T_0}\int_{B^c(0,\, r)}|e^{-\i tH}v_{0,\,k}|^2\,\mathrm dx\,\mathrm dt\leq\frac{ C(T_0)}{k}\rightarrow 0,\,\,\,\,\,\,\text{as}\,\,\,\,k\rightarrow\infty,
\end{align}
which leads to a contradiction. So $(i)$ is true.

Thus we end the proof of Theorem \ref{thm-1126-1}.\qed


\subsection{Proof of Theorem \ref{thm-1126-3}.}\label{subsec5.3} In its proof, we borrow some ideas from \cite{DM} and \cite{WWZZ}.

Suppose that $E$ holds \eqref{equ1.26} for some $T>0$ and $C>0$. Notice that the semigroup $\{e^{-tH}\}_{t\geq 0}$, generated by the Hermite operator $-H$,
can be extended over $\mathbb{C}^+:=\{z\in \mathbb{C} : \Re z>0\}$. Furthermore,
  the kernel associated with  $e^{-zH}$ (over $\mathbb{C}^+$) can be written in the form: {
\begin{align}\label{equ-1027-1.31}
K_{z}(x, y)=\sum_{\alpha\in\mathbb{N}^n}{e^{-(n+2|\alpha|)z}\Phi_{\alpha}(x)\Phi_{\alpha}(y)},
 \;\; z\in \mathbb{C}^+, x\in \mathbb{R}^n, y\in \mathbb{R}^n,
\end{align}}
where $\Phi_{\alpha}(x)$ is given by \eqref{5.6,10-27}. Thanks to the Mehler's formula (see e.g. in \cite[p. 85]{Th}), the series in  \eqref{equ-1027-1.31} can be summed explicitly. More precisely, we have that when $z\in \mathbb{C}^+$ and $x,y\in \mathbb{R}^n$,
\begin{align}\label{equ1.28}
K_{z}(x, y)=(2\pi\sinh {2z})^{-\frac n2}\exp\left(-\frac{\coth{2z}}{2}(|x|^2+|y|^2)+\frac{2}{\sinh{2z}}x\cdot y\right),
\end{align}
where $\coth$ and $\sinh$ are hyperbolic trigonometric functions. From \eqref{equ1.28}, we see that for each fixed $(x, y)\in \mathbb{R}^n\times\mathbb{R}^n$, the kernel  is an analytic function of $z$  over $\mathbb{C}^+$.

We  claim
\begin{align}\label{equ1.31}
|K_{s+\i t}(x,y)|\leq Cs^{-\frac n2}\exp\left(-\frac{s\cdot |x-y|^2}{4(s^2+t^2)}\right)\;\mbox{for all}\; x,y\in\mathbb{R}^n, s>0, t\in [0,T].
\end{align}
To prove \eqref{equ1.31}, we first recall the following  result on analytic function: (It can be found in \cite[Lemma 9]{Da}.)
 {\it Let $F$ be an analytic function on $\mathbb{C}^+$. Suppose that
there is $a_1>0$, $a_2>0$, $\beta\geq 0$ and $\alpha\in (0,1]$ so that
\begin{align*}
|F(re^{\i \theta})|\leq a_1(r\cos\theta)^{-\beta},\,\,\,\,\,\,|F(r)|\leq a_1r^{-\beta}\exp{(-a_2r^{-\alpha})}
\;\;\mbox{for all}\;\;r>0\;\mbox{and}\;|\theta|<\pi/2.
\end{align*}
 Then
\begin{align}\label{equ1-1029-29}
|F(re^{\i \theta})|\leq a_12^{\beta}(r\cos\theta)^{-\beta}\exp{(-\frac{a_2\alpha}{2}r^{-\alpha}\cos{\theta})}\;\;\mbox{for all}\;\;r>0\;\mbox{and}\;|\theta|<\pi/2.
\end{align}}

Now we are in a position  to prove \eqref{equ1.31}. On one hand, when $z=s>0$, one can use the inequalities: $e^{2s}+e^{-2s}\ge 2$ and $\sinh{2s}\ge 2s$, to find  some $C>0$ so that
\begin{align}\label{equ1.29}
K_{s}(x, y)\leq Cs^{-\frac n2}\exp\left(-\frac{|x-y|^2}{2s}\right)\;\;\mbox{for all}\; x,y\in \mathbb{R}^n\;\mbox{and}\;s>0.
\end{align}
(One can also use the fact:  $-\Delta+|x|^2\ge -\Delta$, to get \eqref{equ1.29}, see e.g. in \cite{Si}. ) On the other hand,
by \eqref{equ1.28}, it follows that there is some $C>0$ so that
\begin{align}\label{equ1.30}
|K_{s+\i t}(x, y)|\leq Cs^{-\frac n2}\;\;\mbox{for all}\;\;x,y\in \mathbb{R}^n, s>0\;\mbox{and}\;t\geq 0.
\end{align}
In view of \eqref{equ1.29} and \eqref{equ1.30}, the desired estimate \eqref{equ1.31} follows by applying \eqref{equ1-1029-29} with $a_2=|x-y|^2/2$, $\alpha=1$, $\beta=n/2$ and $\cos\theta=s/\sqrt{s^2+t^2}$.

Next, we arbitrarily fix $y_0\in \mathbb{R}^n$. Let
\begin{eqnarray*}\label{equ1.32}
u_0(x):=K_{1}(x, y_0),\;x\in \mathbb{R}^n.
\end{eqnarray*}
(It is clear that $u_0\in L^2(\mathbb{R}^n)$.)
By the group property,  the solution $v(t, x):= e^{-\i tH}u_0(x)$ ($(t, x)\in \mathbb{R^+}\times \mathbb{R}^n$) is as:
\begin{align}\label{equ1.33}
v(t, x)=\int{K_{\i t}(x, y)K_1(y, y_0)\, \mathrm dy}=K_{1+\i t}(x, y_0),\; (t, x)\in \mathbb{R^+}\times \mathbb{R}^n.
\end{align}
Since the Hermite function $\Phi_{0}(x)=\pi^{-n/4}e^{-\frac{|x|^2}{2}}$ ($x\in \mathbb{R}^n$), given by \eqref{5.6,10-27},
is the normalized eigenfunction of $H=-\Delta+|x|^2$ corresponding to the smallest eigenvalue $\lambda_0=n$, it follows  from the eigenfunction expansion of $u_0(x)$ that
\begin{align}\label{equ1.35}
\int_{\mathbb{R}^n}|u_0(x)|^2\,\mathrm dx=\int{\big|\sum_{\alpha\in\mathbb{N}^n}{e^{-(n+2|\alpha|)}\Phi_{\alpha}(x)\Phi_{\alpha}(y_0)}\big|^2\,\mathrm dx}\ge e^{-2n}|\Phi_{0}(y_0)|^2=Ce^{-|y_0|^2}.
\end{align}
Meanwhile,  by \eqref{equ1.33} and \eqref{equ1.31}, we see that for any  $L>0$,
 \begin{align}\label{equ1.36}
|v(t, x)|\leq Ce^{-\frac{\gamma |x-y_0|^2}{8}}e^{-\frac{\gamma L^2 |\rho(y_0)|^2}{8}},\;\;\mbox{when}\;\;|x-y_0|\ge L\rho(y_0)\;\mbox{and}\;t\in[0,T],
\end{align}
where $\gamma:=\frac{1}{1+T^2}$ and $\rho(y_0)$ is given by \eqref{equ1.27}.

{
Finally,  by \eqref{equ1.35},
\eqref{equ1.26} (where $u$ is replaced by the above $v$) and \eqref{equ1.36}, we see that for any $L>0$,
\begin{align}\label{equ1.37}
e^{-|y_0|^2}&\leq C\int_0^{T}\int_{E\bigcap B(y_0,\, L\rho(y_0)) }{|v(t, x)|^2\,\mathrm dx}\,\mathrm dt+Ce^{-\frac{\gamma L^2|\rho(y_0)|^2}{4}}\int_0^{T}\int_{ B^{c}(y_0,\, L\rho(y_0))}{e^{-\frac{\gamma |x-y_0|^2}{4}}\,\mathrm dx}\,\mathrm dt  \nonumber\\
&\leq C\int_0^{T}\int_{E\bigcap B(y_0,\, L\rho(y_0)) }{|v(t, x)|^2\,\mathrm dx}\,\mathrm dt+CTe^{-\frac{\gamma L^2|\rho(y_0)|^2}{4}}.
\end{align}
In the above, by choosing $L>0$ so that $CTe^{-\frac{\gamma L^2|\rho(y_0)|^2}{4}}<\frac{e^{-|y_0|^2}}{2}$, we obtain, with the help of \eqref{equ1.31}, that
}
\begin{eqnarray}\label{equ1.37}
\frac{e^{-|y_0|^2}}{2}&\leq&C\int_0^{T}\int_{E\bigcap B(y_0,\, L\cdot \rho(y_0))}{e^{-\frac{\gamma|x-y_0|^2}{2}}\,\mathrm dx}\,\mathrm dt\nonumber\\
&\leq &CT\big|E\bigcap B(y_0,\, L\rho(y_0))\big|,
\end{eqnarray}
which leads to \eqref{equ1.27} with $c=\frac{1}{2CT}$. This ends the proof of Theorem \ref{thm-1126-3}.   \qed
\begin{remark}\label{remark5.2,11-18}
  In the 1-dim case, the necessary condition \eqref{equ1.27}  is strictly weaker than the weakly thick property \eqref{def-1-3}.
 This can be seen by three facts as follows: First, if  $E\subset \R$   satisfies \eqref{def-1-3}, then $E$ must satisfy \eqref{equ1.27}. Second, the set $E$, defined by \eqref{equ3.61},
 holds \eqref{equ1.27}. Third, the set $E$, defined by \eqref{equ3.61}, does not satisfy  \eqref{def-1-3}.
 The first fact follows from Theorem \ref{thm-ob-for-HO} and Theorem \ref{thm-1126-3}. The second fact can be proved in the following way: Since
 $$[y-3\rho(y),\, y+3\rho(y)]\supseteq[0, 2\rho(y)]\,\,\,\,\text{for all}\,\,y\in\mathbb{R}, $$
 here  $\rho(y)$ is given by \eqref{equ1.27}.  It follows from \eqref{equ3.61}
 that for some $c>0$,
$$
\big|E\bigcap [y-3\rho(y),\, y+3\rho(y)]\big|\ge \rho(y)^{-\varepsilon}\ge ce^{-|y|^2}\;\;\mbox{for all}\;\;y\in\mathbb{R},
$$
 which shows that $E$ satisfies \eqref{equ1.27}. The third fact follows from \eqref{equ3.61-1801} in Remark \ref{rmk-bound-set}.
\end{remark}

\begin{appendix}
\renewcommand{\appendixname}{Appendix\,\,}

\section{WKB approximate solutions-Proof of Lemma \ref{lem-append}}\label{sec-app}
 According to {\bf Key Observation} in the proof of Theorem \ref{thm-4-suff} ,
each $\varphi_k$ is either even or odd.
We claim: in the case that $\varphi_k$ is even, \eqref{eq-app-2}, as well as
\eqref{eq-app-4}, holds, while in the case that $\varphi_k$ is odd, \eqref{eq-app-3}, as well as
\eqref{eq-app-4}, holds. We only give the proof for the case that $\varphi_k$ is even, while the proof for the second case is very similar. {\it Thus, we will assume, in what follows, that  $\varphi_k$ is even.}

Notice that  the equation \eqref{eq-app-1} has two turning points $x=\pm\mu_k$ (with $\mu_k=\lambda_k^{1/2m}$). Since
$\varphi_k$ is even, we need only focus our studies on $[0, \infty)$ and  the turning point $x=\mu_k$.
 Since $\varphi_k$ has different behaviors for the cases that $x$ is small (compared to $\mu_k$); $x$ is close to $\mu_k$; and  $x$ is large,  it has three different expressions in \eqref{eq-app-4}.

\noindent{\it Case 1. Asymptotic behavior of $\varphi_k$ when $0\leq x<\mu_k-\delta\mu_k^{-\frac{2m-1}{3}}$ (Here, $\delta>0$ is arbitrarily fixed. Notice that for large $k$, we have $\mu_k-\delta\mu_k^{-\frac{2m-1}{3}}>0$.)}

  This case is corresponding to the classical allowed region. Consider the following standard Liouville transform (see e.g. \cite[p. 119]{Ti}):
\begin{equation}\label{eq-app-6}
\begin{cases}
y=S^-(x)=\int_0^x{\sqrt{\mu_k^{2m}-s^{2m}}\,\mathrm ds},\;
x\in \hat\Omega_k:=\big\{x\;:\;|x|\leq \mu_k-\delta\mu_k^{-\frac{2m-1}{3}}\big\}\\[4pt]
w=w(y)=\big(\mu_k^{2m}-x^{2m}\big)^{\frac{1}{4}}
 \varphi_k(x),\;\mbox{with}\; y=S^{-}(x)\in S^{-}(\hat\Omega_k).
\end{cases}
\end{equation}
Applying the above transform to the equation \eqref{eq-app-1} (restricted over $\hat\Omega_k$), we find
\begin{align}\label{eq-app-7}
\frac{d^2w(y)}{dy^2}+w(y)+q_1(\mu_k, y)w(y)=0,\;y\in S^{-}(\hat\Omega_k),
\end{align}
where
\begin{align}\label{eq-app-8}
q_1(\mu_k, y)=\frac{m(2m-1)x^{2m-2}}{2(\mu_k^{2m}-x^{2m})^2}+\frac{5m^2x^{2(2m-1)}}{4(\mu_k^{2m}-x^{2m})^3},
\,\mbox{with}\,y=S^-(x)\in S^{-}(\hat\Omega_k).
\end{align}
Using Duhamel's formula to \eqref{eq-app-7}, we have
\begin{align}\label{eq-app-9}
w(y)=w(0)\cos y- \int_0^y{\sin(y-z)q_1(\mu_k, z)w(z)\,\mathrm dz},\; \;y\in S^{-}(\hat\Omega_k).
\end{align}
By \eqref{eq-app-8}, after some computations, we see that  when  $0\leq x<\mu_k-\delta\mu_k^{-\frac{2m-1}{3}}$,
\begin{align}\label{eq-app-10}
\int_0^y{q_1(\mu_k, z) \mathrm dz}&=\int_0^x{\left(\frac{m(2m-1)t^{2m-2}}{2(\mu_k^{2m}-t^{2m})^\frac32}
+\frac{5m^2t^{2(2m-1)}}{4(\mu_k^{2m}-t^{2m})^{\frac{5}{2}}}\right)\,\mathrm dt} \;\; \;(\mbox{by  changing  variable:} z=S^-(t))\nonumber\\
&\leq C\left(\int_0^{\mu_k-\delta\mu_k^{-{(2m-1)}/{3}}}\frac{t^{2m-2}}{(\mu_k^{2m}-t^{2m})^\frac32}\,\mathrm dt
+\int_0^{\mu_k-\delta\mu_k^{-{(2m-1)}/{3}}}\frac{t^{2(2m-1)}}{(\mu_k^{2m}-t^{2m})^\frac52}\,\mathrm dt\right)\nonumber\\
&\leq C\mu_k^{-\frac{2}{3}(m+1)}+C\leq C.\nonumber\\
\end{align}
Here and in what follows, $C$ stands for a positive constant (independent of $k$) which may vary
in different contexts.
 Moreover, when $0\leq x<\mu_k$, we have
\begin{align}\label{eq-app-11}
\int_0^x\frac{t^{2m-2}}{(\mu_k^{2m}-t^{2m})^\frac32}\,\mathrm dt+\int_0^x\frac{t^{2(2m-1)}}{(\mu_k^{2m}-t^{2m})^\frac52}\,\mathrm dt\leq C(\mu_k^{2m}-x^{2m})^{-\frac12}(\mu_k-x)^{-1}.
\end{align}
By \eqref{eq-app-10}, we can apply Gronwall's inequality in \eqref{eq-app-9} to see
\begin{align}\label{eq-app-12}
|w(y)|\leq C\cdot |w(0)|,\;\;\mbox{when}\; y\in S^{-}(\hat\Omega_k).
\end{align}
 Inserting \eqref{eq-app-12} into \eqref{eq-app-9}, using  \eqref{eq-app-11} and \eqref{eq-app-6}, we obtain
\begin{align}\label{eq-app-13}
\varphi_k(x)=w(0)(\mu_k^{2m}-|x|^{2m})^{-\frac14}(\cos{S^-(x)}+R_k(x)),\,\,\,|x|<\mu_k-\delta\mu_k^{-\frac{2m-1}{3}}
\end{align}
where the error term $R_k(x)$  satisfies
\begin{align}\label{eq-app-14}
|R_k(x)|\leq C(\mu_k^{2m}-x^{2m})^{-\frac12}(\mu_k-x)^{-1},\,\,\,\,\,0\leq x<\mu_k.
\end{align}

Comparing \eqref{eq-app-13} (as well as \eqref{eq-app-14}) with  \eqref{eq-app-2} (as well as
\eqref{eq-app-4}), we see that the remainder in {\it Case 1} is to show that
\begin{eqnarray}\label{A10,11-10}
 w(0)\sim \mu_k^{\frac{m-1}{2}}.
 \end{eqnarray}

 First, we notice that Lemma \ref{lem3.1} (where $\Omega_k$ is replaced by $\hat\Omega_k$) can be applied,
 except for a slight modification on \eqref{equ3.5}. Thus,
 Lemma \ref{lem-930-C} holds for  $C_{\lambda_k}:=w(0)-iw'(0)$
  where
 $w$ is given by \eqref{eq-app-6} (see \eqref{3.22,11.4}). Since $w$ is even, we have $C_{\lambda_k}:=w(0)$. This,
 along with \eqref{equ3.15}, yields
\begin{eqnarray*}
|w(0)|\lesssim\lambda_k^{\frac{m-1}{4m}}=\mu_k^{\frac{m-1}{2}}.
\end{eqnarray*}

Next,  we will use the argument in \cite[Lemma 3.2]{Ya} to prove that
$$
|w(0)|\gtrsim \mu_k^{\frac{m-1}{2}}.
$$
 Indeed, from the following viral identity \cite[p. 66]{CFKS}:
$$
\mu_k^{2m}=\left\langle(-\partial_x^2+x^{2m})\varphi_k,\, \varphi_k\right\rangle_{L^2(\mathbb{R})}
=(m+1)\langle x^{2m}\varphi_k,\, \varphi_k\rangle_{L^2(\mathbb{R})},
$$
we have
\begin{eqnarray*}
\mu_k^{2m}&=(m+1)\displaystyle\int_{\mathbb{R}}{x^{2m}\cdot|\varphi_k(x)|^2\,\mathrm dx}\geq \frac{\mu_k^{2m}(m+1)}{\sqrt{2}} \displaystyle\int_{|x|\ge {\mu_k}/{(2^{\frac{1}{4m}}})}{|\varphi_k(x)|^2\,\mathrm dx}.\nonumber\\
\end{eqnarray*}
 From this  and the fact $\int_0^{\infty}|\varphi_k(x)|^2\d x=1/2$,  we have
\begin{align}\label{eq-app-15}
\int_0^{{\mu_k}/{(2^{\frac{1}{4m}}})}{|\varphi_k(x)|^2\,\mathrm dx}=\frac12-\int_{x\ge {\mu_k}/{(2^{\frac{1}{4m}}})}{|\varphi_k(x)|^2\,\mathrm dx}\ge \frac{m+1-\sqrt{2}}{2(m+1)}>0.
\end{align}
On the other hand,  for ${0\leq x\le {\mu_k}/{(2^{\frac{1}{4m}}})}$, it follows from \eqref{eq-app-13} and \eqref{eq-app-14} that
\begin{align}\label{eq-app-16}
\int_0^{{\mu_k}/{(2^{\frac{1}{4m}}})}{|\varphi_k(x)|^2\,\mathrm dx}&=\int_0^{ {\mu_k}/{(2^{\frac{1}{4m}}})}{\frac{|w(0)(\cos{S^-(x)}+R_k(x))|^2}{\sqrt{\mu_k^{2m}-x^{2m}}}\,\mathrm dx}\nonumber\\
&\leq C|w(0)|^2\cdot \mu_k^{1-m},
\end{align}
which, together with \eqref{eq-app-15}, yields that $|w(0)|\gtrsim \mu_k^{\frac{m-1}{2}}$.
Hence, we end the proof of \eqref{A10,11-10}.

\noindent {\it Case 2. Asymptotic behavior of $\varphi_k$ when $\mu_k-\delta \mu_k^{-\frac{2m-1}{3}}\leq x\leq \mu_k+\delta \mu_k^{-\frac{2m-1}{3}}$ (Here, $\delta>0$ will be
given later.)}

Notice that near the turning point $\mu_k$, the approximation in {\it Case 1} breaks down since the factor $|x^{2m}-\mu_k^{2m}|^{-1/4}$ in \eqref{eq-app-13} goes to infinity when $|x-\mu_k|\rightarrow 0$. The way to pass this barrier is to linearize the potential $x^{2m}$ near the turning point. Indeed, plugging the Taylor's expansion:
$$
x^{2m}=\mu_k^{2m}+2m\mu_k^{2m-1}(x-\mu_k)+\mu_k^{2m-2}(x-\mu_k)^2\cdot T(x-\mu_k),
$$
(where $T(t)=c_0+c_1t+\cdots+c_{2m-2}t^{2m-2}$ is some polynomial of degree $2m-2$) into   \eqref{eq-app-1} yields
\begin{align}\label{equ-app-1003-1}
-\varphi''_k(x)+2m\mu_k^{2m-1}(x-\mu_k)\varphi_k(x)+\mu_k^{2m-2}(x-\mu_k)^2 T(x-\mu_k)\cdot\varphi_k(x)=0.
\end{align}
Using the transform:
\begin{align}\label{eq-app-26}
y=(2m\mu_k^{2m-1})^{\frac13}(x-\mu_k),
\end{align}
and choosing $\delta$ so that
 $$
 0<\delta<\frac{(2m)^{-1/3}}{10},
   $$
   we change the equation \eqref{equ-app-1003-1} into
\begin{align}\label{eq-app-1027-26}
\varphi''_k(y)-y\cdot \varphi_k(y)-q_2(\mu_k, y)\varphi_k(y)=0,\,\,\,\,-1/10< y< 1/10,
\end{align}
where
\begin{align}\label{eq-app-1027-27}
q_2(\mu_k, y)=\mu_k^{-\frac{2(m+1)}{3}}\cdot y^2\cdot T\big((2m\mu_k^{2m-1})^{-\frac13}y\big).
\end{align}

 The idea to deal with \eqref{eq-app-1027-26} is as: the term $q_2(\mu_k, y)$ is small as $k$ is large; when it is ignored, the above equation becomes the standard Airy equation. Thus, it can be solved  explicitly in terms of $Ai(\cdot)$ and $Bi(\cdot)$ (which are two linear independent solutions of the Airy equation: $\psi''(y)-y\psi(y)=0$, $y\in \mathbb{R}$) by the method of variation of parameters (see \cite{Ol}).

With the above idea, we  write the solution of \eqref{eq-app-1027-26} as:
\begin{eqnarray}\label{A17,11-11}
\varphi_k(y)=c_{\mu_k}\cdot\big(Ai(y)+r(y)\big),\,\,\,\,-1/10< y< 1/10.
\end{eqnarray}
Since $Ai''(y)-y\cdot Ai(y)=0$,  the error term $r(\cdot)$ satisfies
\begin{align}\label{eq-app-1027-28}
r''(y)-y\cdot r(y)-q_2(\mu_k, y)(Ai(y)+r(y))=0,\,\,\,\,-1/10< y< 1/10.
\end{align}
By the variation of parameters and by using of the identity:
$$
W(Ai(y), Bi(y)):=Ai(y)Bi'(y)-Bi(y)Ai'(y)=1/\pi,\; y\in \mathbb{R},
$$
we can get from \eqref{eq-app-1027-28} that (see e.g. \cite[p.400]{Ol})
\begin{align}\label{eq-app-1027-29}
r(y)=\pi\int_{-1/10}^y{\big(Bi(y)Ai(v)-Ai(y)Bi(v)\big)q_2(\mu_k, v)\big(Ai(v)+r(v)\big)\,\mathrm dv},
\end{align}
when $-1/10< y< 1/10$.
Since the functions $Ai(\cdot)$ and $Bi(\cdot)$ don't vanish over $[-1/10,\, 1/10]$ (\cite[p.395]{Ol}), we can find some absolute constant $C>0$  so that
\begin{align}\label{eq-app-1028-30}
\frac{1}{C}\leq |Ai(y)|,\,|Bi(y)|\leq C\;\;\mbox{for all}\;\;-1/10< y< 1/10.
\end{align}
This, along with  \eqref{eq-app-1027-27}, yields
\begin{align}\label{eq-app-1028-31}
\int_{-1/10}^y{\left|\big(Bi(y)Ai(v)-Ai(y)Bi(v)\big)q_2(\mu_k, v)\right|\,dv}\leq C\mu_k^{-\frac{2(m+1)}{3}},\,\,\,\,-1/10< y< 1/10.
\end{align}
By \eqref{eq-app-1028-31},  we can apply Gronwall's inequality to \eqref{eq-app-1027-29} to see
\begin{align}\label{eq-app-1028-32}
|r(y)|\leq C\mu_k^{-\frac{2(m+1)}{3}},\,\,\,\,-1/10< y< 1/10.
\end{align}

Now, it follows  from \eqref{A17,11-11} and  \eqref{eq-app-26} that
\begin{align}\label{eq-app-27}
\varphi_k(x)=c_{\mu_k}\cdot
\left(Ai\big((2m\mu_k^{2m-1})^{\frac13}(x-\mu_k)\big)
+r\big((2m\mu_k^{2m-1})^{\frac13}(x-\mu_k)\big)\right),
\end{align}
where $r$ satisfies the estimate \eqref{eq-app-1028-32}.
The remainder is to estimate $c_{\mu_k}$ in \eqref{eq-app-27}. By \eqref{eq-app-1028-30}, one has $Ai(x)\sim 1$, when $|x|\leq 1/10$. Therefore we have
\begin{align}\label{eq-app-28}
\varphi_k(x)\sim c_{\mu_k},\;\;\text{when}\;\;\mu_k-\delta\mu_k^{-\frac{2m-1}{3}}\leq x\leq \mu_k+\delta\mu_k^{-\frac{2m-1}{3}}.
\end{align}
Meanwhile,  if $w$ is given by \eqref{eq-app-6}, then by results in {\it Case 1} and by the fact $|w(0)|\sim\mu_k^{\frac{m-1}{2}}$,
 we see that  $x\nearrow \mu_k-\delta\mu_k^{-\frac{2m-1}{3}}$,
$$
\varphi_k(x)\rightarrow w(0)\left(\mu_k^{2m}-(\mu_k-\delta\mu_k^{-\frac{2m-1}{3}})^{2m}\right)^{-1/4}\sim\mu_k^{\frac{m-2}{6}}.
$$
This, together with the continuity of the function $\varphi_k$, yields that
\begin{align}\label{eq-app-29}
c_{\mu_k}\sim \mu_k^{\frac{m-2}{6}}.
\end{align}
From \eqref{eq-app-28} and \eqref{eq-app-29}, we get \eqref{eq-app-2} for {\it Case 2}.

{\it Case 3. Asymptotic behavior of $\varphi_k$ when $x>\mu_k+\delta\mu_k^{-\frac{2m-1}{3}}$.}

This case is corresponding to the classically forbidden region since the potential energy $V=x^{2m}$ is greater than total energy $\lambda_k=\mu_k^{2m}$. So we use the following transform (instead of \eqref{eq-app-6}):
\begin{equation}\label{eq-app-18}
\begin{cases}
y=S^+(x)=\int_{\mu_k}^x{\sqrt{s^{2m}-\mu_k^{2m}}\,\mathrm ds},\\[4pt]
w=w(y)=(x^{2m}-\mu_k^{2m})^{\frac14}\varphi_k(x),\;\mbox{with}\; y=S^+(x).
\end{cases}
\end{equation}
Under this transform, the equation \eqref{eq-app-1} is as:
\begin{align}\label{eq-app-19}
\frac{d^2w(y)}{dy^2}-w(y)+q_3(\mu_k, y)w(y)=0,
\end{align}
where
\begin{align}\label{eq-app-20}
q_3(\mu_k, y)=\frac{m(2m-1)x^{2m-2}}{2(x^{2m}-\mu_k^{2m})^2}+\frac{5m^2x^{2(2m-1)}}{4(x^{2m}-\mu_k^{2m})^3},\,\,\,\,y=S^+(x).
\end{align}
Then by \eqref{eq-app-20} and a direct computation, one has
$$
\int_{S^+(\mu_k+\delta\mu_k^{-(2m-1)/3})}^{\infty}{q_3(\mu_k, y)\,\mathrm dy}\leq C\int_{\mu_k+\delta\mu_k^{-\frac{2m-1}{3}}}^{\infty}\left(\frac{x^{2m-2}}{(x^{2m}-\mu_k^{2m})^{\frac32}}+\frac{x^{2(2m-1)}}{(x^{2m}-\mu_k^{2m})^{\frac52}}\right){\,\mathrm dx}\leq C.
$$
This shows that the unbounded potential $x^{2m}$ is reduced to an integrable one after the transform \eqref{eq-app-18}. Then, according to some fundamental results in ODE (see e.g. \cite[Theorem 4.1]{BS}), the solution satisfies
\begin{align}\label{eq-app-22}
w(y)=C_{\mu_k}e^{-y}(1+o(1)), \,\,\,\text{as}\,\,\, y\rightarrow+\infty,
\end{align}
and it is uniquely defined by
\begin{align}\label{eq-app-21}
w(y)=C_{\mu_k}e^{-y} - \int_y^{\infty}{\sinh(t-y)q_3(\mu_k, t)w(t)\,\mathrm dt}.
\end{align}
To proceed, observe that by  \eqref{eq-app-22}, one has
$$
|\sinh{(t-y)}\cdot w(t)|\leq C_{\mu_k}\big(e^{t-y}-e^{-(t-y)}\big)e^{-t}\leq C_{\mu_k}e^{-y},\;\;\mbox{when}\;\;t>y.
$$
Plugging this into the integral in \eqref{eq-app-21} and making change of variable $t=S^+(s)$ (notice that $y=S^+(x)$),  one has
\begin{align}\label{eq-app-23}
\left|\int_y^{\infty}{\sinh(t-y)q_3(\mu_k, t)w(t)\,\mathrm dt}\right|&\leq  C_{\mu_k}e^{-y}\cdot\int_y^{\infty}{q_3(\mu_k, t)\,\mathrm dt}\nonumber\\
&\leq   C_{\mu_k}e^{-y}\cdot\int_x^{\infty}{\left(\frac{s^{2m-2}}{(s^{2m}-\mu_k^{2m})^{\frac32}}+\frac{s^{2(2m-1)}}{(s^{2m}-\mu_k^{2m})^{\frac52}}\right)\,\mathrm ds}\nonumber\\
&\leq C_{\mu_k}e^{-y}\cdot(x^{2m}-\mu_k^{2m})^{-\frac12}(x-\mu_k)^{-1},\,\,\,\,x>\mu_k.
\end{align}
Therefore we obtain from \eqref{eq-app-21},  \eqref{eq-app-23} and  \eqref{eq-app-18} that
\begin{align}\label{eq-app-24-1117}
\varphi_k(x)=C_{\mu_k}(x^{2m}-\mu_k^{2m})^{-\frac14}e^{-S^+(x)}\big(1+R_k(x)\big),
\;\;\mbox{when}\;\;x>\mu_k+\delta\mu_k^{-\frac{2m-1}{3}},
\end{align}
and  the reminder term $R_k(x)$  satisfies
\begin{align}\label{eq-app-25}
|R_k(x)|\leq C(x^{2m}-\mu_k^{2m})^{-\frac12}(x-\mu_k)^{-1},
\;\;\mbox{when}\;\;x>\mu_k+\delta\mu_k^{-\frac{2m-1}{3}}.
\end{align}

We next claim that
\begin{eqnarray}\label{A34,11-12}
C_{\mu_k}\sim\mu_k^{\frac{m-1}{2}}
\end{eqnarray}
Indeed, according to \eqref{eq-app-18} and \eqref{eq-app-25},
there is an absolute constant $C>0$ such that
$$
0<S^+(\mu_k+\delta\mu_k^{-\frac{2m-1}{3}})\leq \mu_k^{-\frac{2m-1}{3}}\cdot \sqrt{(\mu_k+\delta\mu_k^{-\frac{2m-1}{3}})^{2m}-\mu_k^{2m}}\leq C,
$$
and
$$
R_k(\mu_k+\delta\mu_k^{-\frac{2m-1}{3}})\leq C.
$$
By these, we can use \eqref{eq-app-24-1117} and the continuity of  $\varphi_k$  near the point $\mu_k+\delta\mu_k^{-\frac{2m-1}{3}}$ to find that
 when $x\searrow\mu_k+\delta\mu_k^{-\frac{2m-1}{3}}$,
$$
\varphi_k(x)\rightarrow C_{\mu_k}\left((\mu_k+\delta\mu_k^{-\frac{2m-1}{3}})^{2m}
-\mu_k^{2m}\right)^{-1/4}\sim\mu_k^{\frac{m-2}{6}},
$$
which leads to \eqref{A34,11-12}.

Finally, by \eqref{eq-app-24-1117}, \eqref{eq-app-25} and \eqref{A34,11-12}, we get
\eqref{eq-app-2} and
\eqref{eq-app-4} for {\it Case 3}. This ends  the proof of Lemma \ref{lem-append}. \qed

\begin{remark}
The proof of Lemma \ref{lem3.1}, provided in \cite{Ya},
 is in the same spirit of  the above {\it Case 1}.
This is because the region $\Omega_{\mu_k}=\{x\in \mathbb{R},\,\,\,V(x)\le \frac{\mu_k}{2}\}$ in the consideration belongs to the classically allowed region in {\it Case 1}.
Since the potential in  Lemma \ref{lem3.1} satisfies the \textbf{Condition (H)},  \cite{Ya} uses transform \eqref{equ3.4} instead of \eqref{eq-app-6}.
Moreover, the  \textbf{Condition (H)} is sufficient to obtain the sharp bounds of the constant $w(0)$.
\end{remark}

\end{appendix}

\section*{Acknowledgements}
The authors would like to thank  Dr. Yubiao Zhang for very useful discussions and for constructive comments on the time optimal observability of Hermite Schr\"{o}dinger equations.

S. Huang was supported by the National Natural Science Foundation of China under grants 11801188
and 11971188.

G. Wang was partially supported by the National Natural Science Foundation of China under grants 11971022 and 11926337.

M. Wang was supported by the National Natural Science Foundation of China under grant No. 11701535

\bibliographystyle{IEEEtran}

\end{document}